\documentclass[a4paper,10pt]{article}
\usepackage[utf8]{inputenc}

\usepackage{wrapfig}
\usepackage{amsmath} 
\usepackage{amsthm} 
\usepackage{amssymb} 
\usepackage{enumerate} 
\usepackage{esint} 
\usepackage{pgf,tikz} 
\usetikzlibrary{arrows} 
 \usepackage{yfonts} 
 \usepackage{mathrsfs} 
 \usepackage{mathabx} 
 \usepackage{graphicx}
\usepackage{caption}
\usepackage{subcaption}

\definecolor{wwffqq}{rgb}{0.4,1,0}
\definecolor{yqqqqq}{rgb}{0.5,0,0}
\definecolor{qqqqcc}{rgb}{0,0,0.8}
\definecolor{zzttqq}{rgb}{0.6,0.2,0}
\definecolor{ffdxqq}{rgb}{1,0.84,0}
\definecolor{qqttzz}{rgb}{0,0.2,0.6}
\definecolor{qqqqff}{rgb}{0,0,1}
\definecolor{uuuuuu}{rgb}{0.27,0.27,0.27}
\definecolor{verdfort}{rgb}{0.5,0.8,0.5}

\newcommand*\circled[1]{\tikz[baseline=(char.base)]{
            \node[shape=circle,draw,inner sep=2pt] (char) {#1};}}

\newcommand*\squared[1]{\tikz[baseline=(char.base)]{
            \node[shape=rectangle,draw,inner sep=2pt] (char) {#1};}}

\newcommand{\C}{{\mathbb C}}       
\newcommand{\R}{{\mathbb R}}       
\newcommand{\N}{{\mathbb N}}       
\newcommand{\DDD}{{\mathbb D}}
\newcommand{\diam}{{\rm diam}}
\newcommand{\dist}{{\rm dist}}
\newcommand{\Dist}{{\rm D}}
\newcommand{\rf}[1]{{(\ref{#1})}}
\newcommand{\supp}{{\rm supp}}

\newcommand{\norm}[1]{{\left\| {#1} \right\|}}

\newtheorem{theorem}{Theorem}
\newtheorem*{theorem*}{Theorem}
\newtheorem{lemma}[theorem]{Lemma}
\newtheorem{claim}[theorem]{Claim}

\newtheorem{klemma}[theorem]{Key Lemma}

\newtheorem*{coro*}{Corollary}
\newtheorem{propo}[theorem]{Proposition}
\newtheorem{definition}[theorem]{Definition}

\newtheorem{example}[theorem]{Example}
\newtheorem{rem}[theorem]{Remark}

\numberwithin{subsection}{section}
\numberwithin{theorem}{section}
\numberwithin{equation}{section}
\numberwithin{figure}{section}

\usepackage[affil-it]{authblk}

\title{A T(P) theorem for Sobolev spaces on domains}

\author{Mart\'i Prats and Xavier Tolsa
\thanks{MP (Departament de Ma\-te\-m\`a\-ti\-ques, Universitat Aut\`onoma de Bar\-ce\-lo\-na, Catalonia): \texttt{mprats@mat.uab.cat}.
XT (Instituci\'{o} Catalana de Recerca i Estudis Avan\c{c}ats (ICREA) and Departament de
Ma\-te\-m\`a\-ti\-ques, Universitat Aut\`onoma de Bar\-ce\-lo\-na, Catalonia): \texttt{xtolsa@mat.uab.cat}}}

\begin{document}
\maketitle

\begin{abstract} Recently, V. Cruz, J. Mateu and J. Orobitg have proved a T(1) theorem for the Beurling transform in the complex plane. It asserts that given $0<s\leq1$, $1<p<\infty$ with $sp>2$ and a Lipschitz domain $\Omega\subset \mathbb{C}$, the Beurling transform $Bf=- {\rm p.v.}\frac1{\pi z^2}*f$ is bounded in the Sobolev space $W^{s,p}(\Omega)$ if and only if $B\chi_\Omega\in W^{s,p}(\Omega)$.

In this paper we obtain a generalized version of the former result valid for any $s\in \mathbb{N}$ and for a larger family of Calder\'on-Zygmund operators in any ambient space $\mathbb{R}^d$ as long as $p>d$. In that case we need to check the boundedness not only over the characteristic function of the domain, but over a finite collection of polynomials restricted to the domain. Finally we find a sufficient condition in terms of Carleson measures for $p\leq d$. In the particular case $s=1$, this condition is in fact necessary, which yields a complete characterization.
\end{abstract}

\section{Introduction}
The aim of the present article is to find necessary and sufficient conditions on certain singular integral operators to be bounded in the Sobolev space of a Lipschitz domain.  

An operator $T$ defined for $f\in L^1_{loc}(\R^d)$ and $x\in \R^d\setminus \supp(f)$ as 
$$Tf(x) = \int_{\R^d} K(x-y)f(y) dy,$$
is called a \emph{smooth convolution Calder\'on-Zygmund operator of order $n$} if it is bounded in the Sobolev space $W^{n,p}(\R^d)$ (the space of $L^p$ functions with distributional derivatives up to order $n$ in $L^p$) for every $1<p<\infty$ and its kernel $K$ satisfies
\begin{equation*}
|\nabla^jK(x)|\leq \frac{ C}{|x|^{d+j}}
\end{equation*}
 for $0\leq j \leq n$ (see Section \ref{secpreli} for more details). In the present article we deal with some properties of the operator $T$ truncated to a domain $\Omega$, defined as $T_\Omega(f)=\chi_\Omega \, T(\chi_\Omega \, f)$. 
 
In the complex plane, for instance, the \emph{Beurling transform}, which is defined as the principal value 
\begin{equation}\label{eqbeurling}
Bf(z):=-\frac{1}{\pi}\lim_{\varepsilon\to0}\int_{|w-z|>\varepsilon}\frac{f(w)}{(z-w)^2}dm(w),
\end{equation}
is a smooth convolution Calder\'on-Zygmund operator of any order with kernel
$K(z)=-\frac1{\pi \, z^2}.$

In the recent article \cite{cmo}, V\'ictor Cruz, Joan Mateu and Joan Orobitg, seeking for some results on the Sobolev smoothness of quasiconformal mappings proved the next theorem.
\begin{theorem*}[\cite{cmo}]
Let $\Omega$ be a bounded $C^{1+\varepsilon}$ domain (i.e. a Lipschitz domain with parameterizations of the boundary in $C^{1+\varepsilon}$) for a given $\varepsilon>0$, and let $1<p<\infty$ and $0<s\leq1$ such that $sp>2$. Then the truncated Beurling transform $B_\Omega$ is bounded in the Sobolev space $W^{s,p}(\Omega)$ if and only if $B(\chi_\Omega)\in W^{s,p}(\Omega)$.
\end{theorem*}

This was proved in fact for a wider class of even Calder\'on-Zygmund operators in the plane. Using a result in \cite{mov}, one can see that, if $\varepsilon>s$ and $\Omega$ is a $C^{1+\varepsilon}$ domain then $B\chi_\Omega\in W^{s,p}(\Omega)$, so we have that, assuming  the conditions in the previous theorem for $\Omega$, $s$ and $p$, one always has the Beurling transform bounded in $W^{s,p}(\Omega)$. Using this result, in \cite{cmo} the authors deduce the next remarkable theorem that we state here as a corollary.
\begin{coro*}[\cite{cmo}]
Let $\Omega$, $s$ and $p$  be as in the previous theorem with the restriction $\varepsilon>s$. Given a function $\mu$ such that $\supp(\mu)\subset \bar\Omega$ and $\norm{\mu}_\infty< 1$, consider the Beltrami equation
$$\bar\partial\phi(z)=\mu(z)\partial\phi(z),$$
and consider its principal solution $\phi(z)=z+C(h)(z)$, where $C$ stands for the Cauchy transform.
If $\mu\in W^{s,p}(\Omega)$, then  $h\in W^{s,p}(\Omega)$.
\end{coro*}

In this paper, we consider the extension of the theorem above to higher orders of smoothness $s$ and other ambient spaces $\R^d$. We have restricted ourselves to the study of the classical Sobolev spaces, where the smoothness is a natural number, so we denote it by $n$. The first result of the present article is the next theorem.

\begin{theorem}\label{theoTP}
Let $\Omega\subset \R^d$ be a Lipschitz domain, $T$ a smooth convolution Calder\'on-Zygmund operator of order $n \in \N$ and $p>d$. Then the following statements are equivalent:
\begin{enumerate}[a)]
\item The truncated operator $T_\Omega$ is bounded in $W^{n,p}(\Omega)$.
\item For every polynomial $P$ of degree at most $n-1$, we have that $T_\Omega(P)\in W^{n,p}(\Omega)$.
\end{enumerate}
\end{theorem}
The notation is explained in Section \ref{secpreli}. Note that we do not assume the kernel to be even. This result reminds us the results by Rodolfo H. Torres in \cite{torres}, where the characterization of some generalized Calder\'on-Zygmund operators which are bounded in the homogeneous Triebel-Lizorkin spaces in $\R^d$ is given in terms of its behavior over polynomials. Let us also remark that in \cite{vahakangas} Antti V. V\"ah\"akangas obtained some T1 theorem for weakly singular integral operators on domains. Roughly speaking, he showed the image of the characteristic function being in a certain BMO-type space to be equivalent to the boundedness of $T_\Omega: L^p(\Omega) \to \dot W^{m,p}(\Omega)$ where $m$ is the degree of the singularity of T's kernel.

In 2009, V\'ictor Cruz and Xavier Tolsa found a sufficient condition weaker than $\varepsilon>s$ for the validity of the corollary. Namely, they proved in \cite{ct} that if $\Omega\subset \C$ is a Lipschitz domain and its unitary outward normal vector $N$ is in the Besov space $B^{s-1/p}_{p,p}(\partial \Omega)$ (following the notation in \cite{triebel}), then one has $B(\chi_\Omega)\in W^{s,p}(\Omega)$. Furthermore, the parameterizations of the boundary are in $B^{s-1/p+1}_{p,p}(\partial \Omega) \subset C^{1+\epsilon}(\partial \Omega)$ if $sp>2$ (see \cite[Section 2.7.1]{triebel}), so one can use the result in \cite{cmo}, leading to the boundedness of the Beurling transform. Xavier Tolsa proved in \cite{tolsasharp} that this geometric condition is necessary when the Lipschitz constants of $\partial \Omega$ are small. The result in \cite{ct} can be extended to $n\geq 2$ but it is out of reach of the present article. This will be the subject of a forthcoming paper by us.

In Section \ref{seccarleson} we define the shadows $\mathbf{Sh}(x)$ and $\widetilde{\mathbf{Sh}}(x)$ for every point $x$ in a Lipschitz domain $\Omega$ close enough to $\partial\Omega$. Those shadows can be understood as Carleson boxes of the domain. We say that a positive and finite Borel measure $\mu$ is a $p$-\emph{Carleson measure} if for every $a \in \Omega$ and close enough to the boundary,
\begin{equation}\label{eqcontinucarleson}
\int_{\widetilde {\mathbf{Sh}}(a)} \dist(x, \partial\Omega)^{(d-p)(1-p')}(\mu({\mathbf{Sh}}(x)\cap {\mathbf{Sh}}(a)))^{p'} \frac{dx}{\dist(x, \partial\Omega)^d}\leq C \mu({\mathbf{Sh}}(a)).
\end{equation}
N. Arcozzi, R. Rochberg and E. Sawyer proved in \cite{ars} that in the case when $\Omega$ coincides with the unit disk $\DDD\subset\C$, the measure $\mu$ is $p$-Carleson if and only if the trace inequality
$$\int_\DDD |f|^p \, d\mu \leq C|f(0)|^p + C\int_\DDD |f'|^p \, dm$$
holds for any holomorphic function $f$ on $\DDD$. It turns out that the notion of $p$-Carleson measure is also essential for the characterization of the boundedness of Calder\'on-Zygmund operators of order $n$ in $W^{n,p}(\Omega)$ when $1<p\leq d$ as our next theorem shows.
\begin{theorem}\label{theocarleson}
Let $T$ be a smooth convolution Calder\'on-Zygmund operator of order $n$, and consider a Lipschitz domain $\Omega$ and $1<p\leq d$. If the measure $|\nabla^n T_\Omega P(x)|^p dx$ is a $p$-Carleson measure for every polynomial $P$ of degree at most $n-1$, then $T_\Omega$ is a bounded operator on $W^{n,p}(\Omega)$.
\end{theorem}
This condition is in fact necessary for $n=1$:
\begin{theorem}\label{theoTb}
Let $T$ be a smooth convolution Calder\'on-Zygmund smooth operator of order 1, and consider a Lipschitz domain $\Omega$ and $1<p<\infty$. The following statements are equivalent:
\begin{enumerate}
\item $T_\Omega$ is a bounded operator on $W^{1,p}(\Omega)$.
\item The measure $|\nabla T \chi_\Omega(x)|^p dx$ is a $p$-Carleson measure for $\Omega$.
\end{enumerate}
\end{theorem}

\begin{example}
Those theorems can be used to prove the boundedness of $B_\DDD$ in $W^{n,p}(\DDD)$ for any $n\in \N$ and $p>1$ in one stroke. Indeed, given any multiindex $\lambda =(\lambda_1, \lambda_2)$, consider $P_\lambda(z)=z^\lambda=z^{\lambda_1}\overline{z}^{\lambda_2} $. In \cite[page 96]{aim} the authors find a function $f\in W^{1,p}(\C)$ for $p$ big such that $\bar \partial f= \chi_\DDD$ and then using that $B(\bar\partial f)=\partial f$ they deduce who is $B\chi_\DDD$. Using the same procedure, one can see that
\begin{itemize}
\item if $\lambda_1=0$, then $B_\DDD(P_\lambda)(z)= C_\lambda z^{-\lambda_2-2}\chi_{\DDD^c}(z)$,
\item if $0<\lambda_1<\lambda_2+1$, then $B_\DDD(P_\lambda)(z)=C_\lambda^1 z^{\lambda+(-1,1)}\chi_\DDD(z) + C_\lambda^2 z^{\lambda_1-\lambda_2-2}\chi_{\DDD^c}(z)$,
\item if $\lambda_1=\lambda_2+1$, then $B_\DDD(P_\lambda)(z)=C_\lambda z^{\lambda+(-1,1)}\chi_\DDD(z)$,
\item if $\lambda_1>\lambda_2+1$, then $B_\DDD(P_\lambda)(z)=\left(C_\lambda^1 z^{\lambda+(-1,1)} + C_\lambda^2 z^{\lambda_1-\lambda_2-2}\right)\chi_\DDD(z)$,
\end{itemize}
with constants depending only on $\lambda$. Summing up, for any polynomial $P$ of degree $n-1$, its transform $B_\DDD P$ agrees with a polynomial of degree smaller or equal than $n-1$ in $\DDD$ so $\nabla^n B_\DDD P_\lambda (z)= 0$ for $z\in \DDD$. Thus, the sufficient conditions of Theorems \ref{theoTP} and \ref{theocarleson} are satisfied.
\end{example}

\begin{example}
For a negative example, consider a square $Q$ in the complex plane with a corner at $\omega$. In that case, one can see that $B(\chi_Q)(z)$ is expressed as a sum of logarithms \cite[(4.122)]{aim}. Since $|\partial B(\chi_Q)(z)|\approx |z-\omega|^{-1}$ when $z$ is close enough to $\omega$, it follows that $B(\chi_Q) \notin W^{1,p}(Q)$ for $p\geq 2$ and, thus, $B_Q$ is not bounded in $W^{1,p}(Q)$ for $p\geq2$. By the same token,  for $n\geq2$ one has $|\partial^n B(\chi_Q)(z)|\approx |z-\omega|^{-n}$ and therefore $B_Q$ is not bounded in $W^{n,p}(Q)$ for any $p>1$. However, since  $B(\chi_Q)$ is analytic, one can see with some effort that when $n=1$ and $p<2$, then $\mu(z)=|\nabla B \chi_Q(z)|^p$ is a $p$-Carleson measure. Using Theorem \ref{theocarleson}, this leads to the boundedness of $B_Q$ in $W^{1,p}(Q)$ for $1<p<2$.

The question arises whether is there any Lipschitz domain $\Omega$ such that $B_\Omega$ is not bounded in $W^{1,p}(\Omega)$ for $p<2$. We refer the reader to \cite{tolsasharp} to find the tools to answer this question in the affirmative.
\end{example}

The plan of the paper is the following. In Section \ref{secpreli} we begin by stating some remarks and definitions and then we cite some results that we will use. In Section \ref{secwhitney} we define an oriented Whitney covering and we discuss about its properties. To end with the preliminaries, we present some approximating polynomials for a given function $f\in W^{n,p}(\Omega)$ in Section \ref{secpoly}. These polynomials will be the cornerstone of the proof of Theorems \ref{theoTP} and \ref{theocarleson}. Before we prove these theorems, we devote the rather technical Section \ref{secderivatives} to show the existence of weak derivatives of $T_\Omega f$ in $\Omega$ as long as $f\in W^{n,p}(\Omega)$. The expert reader may skip it.  In Section \ref{seclemma} we prove a Key Lemma which is the first step toward the proofs of Theorems \ref{theoTP}, \ref{theocarleson} and  \ref{theoTb}. Afterwards we prove Theorem \ref{theoTP} in Section \ref{secmain}, Theorem \ref{theocarleson} in Section \ref{seccarleson} and Theorem \ref{theoTb} in Section \ref{secmain2}. Finally, in Section \ref{secfinalremarks} we sketch an alternative argument for Theorem \ref{theoTb} in the planar case using complex analysis.

\section{Notation and well-known facts}\label{secpreli}
Along this paper $m$ stands for the Lebesgue measure and $\mathcal{H}^k$ for the $k$-th dimensional Hausdorff measure. We write $dx$ for $dm(x)$ when integrating on subsets of $\R^d$ with respect to the Lebesgue measure if there is no risk of confusion.

We call $\mathcal{P}^n$ the vector space of polynomials of degree smaller or equal than $n$ (in $\R^d$).

The polynomials and derivatives will be written with the multiindex notation. For every multiindex $\alpha\in\N^d$ (where we assume the natural numbers to include the $0$), $\alpha=(\alpha_1,\cdots,\alpha_d)$, we define its modulus as $|\alpha|=\sum_{j=1}^d \alpha_j$ and its factorial $\alpha!:=\prod_{j=1}^d \alpha_j!$, leading to the usual definitions of combinatorial numbers. For two multiindices $\alpha, \beta\in \N^d$ we write $\alpha\leq\beta$ whenever $\alpha_i\leq \beta_i$ for $1\leq i \leq d$, and we write $\alpha<\beta$ if $\alpha\leq\beta$ and $\alpha\neq\beta$. For $x\in \R^d$ let $x^\alpha:=\prod_{j=1}^d x^{\alpha_j}_j$ and for $\phi \in C^\infty_c$ (infinitely many times differentiable with compact support), let $D^\alpha \phi := \frac{\partial^{|\alpha|}}{\partial x_1^{\alpha_1}\cdots \partial x_d^{\alpha_d}}\phi$.

In general, for any open set $U$, and every distribution $f \in \mathcal{D}'(U)$, the $\alpha$ \emph{distributional derivative} of $f$ is defined by
$$ \langle D^\alpha f , \phi \rangle : =(-1)^{|\alpha|}\langle f, D^\alpha \phi \rangle \mbox{\,\,\, for every } \phi\in C^\infty_c(U).$$
If the distribution is regular, that is $D^\alpha f \in L^1_{loc}$, we say it is a \emph{weak derivative} in $U$. We write $|\nabla^n f|=\sum_{|\alpha|=n}|D^\alpha f|$.

We say that $f\in L^p(U)$ is in the \emph{Sobolev space} $W^{n,p}(U)$ if it has weak derivatives up to order $n$ and $D^\alpha f \in L^p(U)$ for $|\alpha|\leq n$. We say that $f\in W^{n,p}_{loc}(U)$ if those derivatives are in the space $L^p_{loc}(U)$ instead. 
We will use the norm 
$$\norm{f}_{W^{n,p}(U)} = \sum_{|\alpha| \leq n}\norm{D^\alpha f}_{L^p(U)}.$$
For Lipschitz domains, it is enough to consider the higher order derivatives and the function itself, 
$$\norm{f}_{W^{n,p}(U)} \approx \norm{f}_{L^p(U)}+\norm{\nabla^n f}_{L^p(U)}$$
(see \cite[4.2.4]{triebel}).

\begin{definition}\label{defCZK}
We say that a measurable function $K \in W^{n,1}_{loc}(\R^d \setminus \{0\})$ is a \emph{smooth convolution Calder\'on-Zygmund kernel of order $n$} if 
\begin{equation*}
|\nabla^jK(x)|\leq \frac{ C_K}{|x|^{d+j}} \mbox{\,\,\,\, for $x\neq 0$ and $0\leq j \leq n$},
\end{equation*}
for a positive constant $C_K$ and that kernel can be extended to a tempered distribution $W_K$ in $\R^d$ in the sense that for every Schwartz function $\phi \in \mathcal{S}$ with $0\notin \supp(\phi)$,  one has
$$\langle W_K, \phi \rangle=(K*\phi)(0).$$
 \end{definition}
%

We will use the classical notation $\widehat f$ for the Fourier transform of a given Schwartz function,
$$\widehat f (\xi)=\int_{\R^d} e^{-2\pi i x\cdot \xi} f(x) dx,$$ 
and $\widecheck f$ will denote its inverse.
It is well known that the Fourier transform can be extended to the whole space of tempered distributions by duality and it induces an isometry in $L^2$ (see for example \cite[Chapter 2]{grafakos}).

\begin{definition}\label{defCZO}
We say that an operator $T:\mathcal{S} \to \mathcal{S'}$ is a \emph{smooth convolution Calder\'on-Zygmund operator of order $n$} with kernel $K$ if $K$ is a smooth convolution Calder\'on-Zygmund kernel of order $n$ such that $\widehat {W_K} \in L^1_{loc}$, $T$ is defined as
$$T\phi=W_K*\phi := \left(\widehat {W_K} \cdot \widehat \phi\right)\widecheck{\,}$$
for every $\phi \in \mathcal{S}$, and $T$ extends to an operator bounded in $L^p$ for every $1<p<\infty$.
 \end{definition}

One can see using the results in \cite[Chapter IV]{steinpetit} and \cite[Chapter 4]{grafakos}, for instance, that this boundedness property is equivalent to having $\widehat {W_K} \in L^\infty$. 

It is a well-known fact that the Schwartz class is dense in $L^p$ for $p<\infty$. Thus, if $f\in L^p$ and $x\notin \supp(f)$, then
$$Tf(x) = \int K(x-y)f(y) dy.$$

\begin{example}
In the complex plane, the Beurling transform \rf{eqbeurling} is a smooth convolution Calde-r\'on-Zygmund operator of any order associated to the kernel $K(z)=-\frac1{\pi \, z^2}$ and its multiplier is $\widehat{W_K}(\xi)=\frac{\bar\xi}{\xi}$.
Thus, the Beurling transform is an isometry in $L^2$.
\end{example}

For any cube $Q$ we write $\ell(Q)$ for its side-length. Given $r\in \R$ we write $rQ$ for the cube concentric with $Q$  and side length $r \ell(Q)$.

\begin{definition}\label{defLipschitz}
Let $\Omega\subset \R^d$ be a domain (open and connected). We say that a cube $\mathcal{Q}$ with side-length $R>0$ and center $x\in\partial\Omega$ is an \emph{$R$-window} of the domain if it induces a local parameterization of the boundary, i.e. there exists a continuous function $A_\mathcal{Q}:\R^{d-1}\to\R$ such that, after a suitable rotation that puts all the faces of $\mathcal{Q}$ parallel to the coordinate axes, 
$$\Omega \cap 2\mathcal{Q}=\{(y',y_d)\in (\R^{d-1}\times \R)\cap2\mathcal{Q} : y_d > A_\mathcal{Q}(y')\}$$
(we use the double cube $2\mathcal{Q}$ in order to ensure that the central point of the upper face of $\mathcal{Q}$ is far from the boundary of $\Omega$).

 We say that a bounded domain $\Omega$ is a \emph{$(\delta,R)$-Lipschitz domain} if for each $x\in\partial \Omega$ there exists an $R$-window $\mathcal{Q}$ centered in $x$ with $A_\mathcal{Q}$ Lipschitz with a uniform bound $\norm{\nabla A_\mathcal{Q}}_\infty<\delta$.
 
 We say that an unbounded domain $\Omega$ is a \emph{special $\delta$-Lipschitz domain} if there exists a Lipschitz function $A$ such that
 $\norm{\nabla A}_\infty<\delta$ and
 $$\Omega=\{(y',y_d)\in \R^{d-1}\times \R : y_d > A(y')\}.$$
 \end{definition}

With no risk of confusion, we will forget often about the parameters $\delta$ and $R$ and we will talk in general of Lipschitz domains and windows without further explanations.

In Section \ref{secmain2} we will solve a Neumann problem by means of the Newton potential: given an integrable  function with compact support $g\in L^1_0(\R^d)$, its Newton potential is
\begin{equation}\label{eqdefpotencialnewton}
N {g} (x)= \int \frac{|x-y|^{2-d}}{(2-d)w_d} {g}(y)\,dy \mbox{\quad if $d> 2$, \,\,  \, \,\,} N {g}(x)= \int \frac{\log{|x-y|}}{2\pi} {g} (y)\,dy \mbox{\quad if d=2,} 
\end{equation}
where $w_d$ stands for the surface measure of the unit sphere in $\R^d$. Recall that the gradient of $N {g}$ is the $(d-1)$-dimensional Riesz transform of ${g}$, 
\begin{equation*}
\nabla N {g}(x)= R^{(d-1)}{g}(x)= \int \frac{x-z}{w_d |x-z|^{d}} {g}(z) \, dz. 
\end{equation*}
It is well known that $\Delta N {g}(x)= {g}(x)$ for $x\in \R^d$ (see \cite[Theorem 2.21]{folland} for instance).

We recall now two results that we will use every now and then. The first is the Leibnitz' Formula, which states that for $f\in W^{n,p}(\Omega)$ and $|\alpha|\leq n$, if $\phi\in C^\infty_c(\Omega)$, then $f \cdot \phi \in W^{n,p}(\Omega)$ and
\begin{equation}\label{eqleibnitz}
D^\alpha (f\cdot \phi)=\sum_{\beta\leq\alpha} \binom{\alpha}{\beta} D^\beta \phi \, D^{\alpha-\beta} f 
\end{equation} 
(see, for example, \cite[5.2.3]{evans}).

The second is the Sobolev Embedding Theorem for Lipschitz domains (see \cite[Theorem 4.12, Part II]{adams}), which says in particular that for each Lipschitz domain $\Omega$ and every $p>d$, we have the continuous embedding 
of the Sobolev space $W^{1,p}(\Omega)$ into the H\"older space $C^{0,1-\frac{d}{p}}(\overline{\Omega})$. Recall that 
$$\norm{f}_{C^{0,s}(\overline{\Omega})}=\norm{f}_{L^\infty(\overline\Omega)}+\sup_{\substack{x,y\in \overline{\Omega}\\x \neq y}} \frac{|f(x)-f(y)|}{|x-y|^s}.$$

\section{Oriented Whitney covering}\label{secwhitney}
Along this section we consider $\Omega$ to be a fixed $(\delta,R)$-Lipschitz domain. We also consider a given dyadic grid of semi-open cubes in $\R^d$.

\begin{definition}\label{defwhitney}
 We say that a collection of cubes $\mathcal{W}$ is a \emph{Whitney covering} of $\Omega$ if
\begin{enumerate}[W1.]
\item The cubes in $\mathcal{W}$ are dyadic.
\item\label{itemWd} The cubes have pairwise disjoint interiors.
\item\label{itemWCovers} The union of the cubes in $\mathcal{W}$ is $\Omega$.
\item \label{itemboundeddistance} There exists a constant $C_\mathcal{W}$ such that $C_\mathcal{W} \ell(Q) \leq \dist(Q, \partial\Omega)\leq 4C_\mathcal{W} \ell(Q)$.
\item\label{itemWneighbor} Two neighbor cubes $Q$ and $R$ (i.e. $\bar Q \cap \bar R \neq \emptyset$, $Q\neq R$) satisfy $\ell(Q)\leq 2\ell(R)$.
\item\label{itemWfinitesuper} The family $\{10Q\}_{Q\in\mathcal{W}}$ has finite superposition, that is $\sum_{Q\in\mathcal{W}} \chi_{10Q}\leq C$.
\end{enumerate}
\end{definition}

We do not prove here the existence of such a covering because this kind of covering is well known and widely used in the literature.

Recall that we say that $\mathcal{Q}$  is an  $R$-window of $\Omega$ if it is a cube centered in $\partial \Omega$, with side-length $R$ inducing a Lipschitz parameterization of the boundary (see Definition \ref{defLipschitz}). 
We can choose a number $N \approx \mathcal{H}^{d-1}(\partial \Omega)/R^{d-1}$ and a collection of windows $\{\mathcal{Q}_k\}_{k=1}^N$ such that 
\begin{equation}\label{eqwindowscontain}
\partial \Omega \subset \bigcup_{k=1}^N \delta_1 \mathcal{Q}_k,
\end{equation}
where ${\delta_1}<\frac14$ is a value to fix later (in Remark \ref{remabove}). 

Each window $\mathcal{Q}_k$ is associated to a parameterization $A_k$ in the sense that, after a rotation,
 $$\Omega \cap 2 \mathcal{Q}_k=\{(y',y_d)\in (\R^{d-1}\times \R)\cap 2\mathcal{Q}_k : y_d > A_k(y')\}.$$
Thus, each $\mathcal{Q}_k$ induces a \emph{vertical} direction, given by the eventually rotated $y_d$ axis. 
The following is an easy consequence of the previous statements and the fact that the domain is Lipschitz:
\begin{enumerate}[W1.]
\setcounter{enumi}{6}
\item\label{itemWvertical} The number of Whitney cubes in $\mathcal{Q}_k$ with the same side-length intersecting a given vertical line is bounded by a constant depending only on the Lipschitz character of $\Omega$, where the ``vertical'' direction is the one induced by the window.
\end{enumerate}

This is the last property of the Whitney cubes we want to point out. Next we define paths connecting Whitney cubes. First, we use that the notion of vertical direction allows us to say that one cube is above another one even if the faces of the Whitney cubes are not parallel to the faces of $\mathcal{Q}_k$.
\begin{definition}\label{defabove}
We say that a cube $S$ is \emph{above} $Q$ with respect to $\mathcal{Q}_k$ if $Q,S\subset \mathcal{Q}_k$, there is a line parallel to the vertical direction induced by $\mathcal{Q}_k$ intersecting the interior of both cubes and there exists a point $x\in S$ such that for every $y\in Q$, $x_d>y_d$ in local coordinates. 
\end{definition}

We distinguish the cubes in the central region from those which are close to the boundary of the domain.
\begin{definition}
We say that $Q$ is \emph{central} if $\sup_{x\in Q}\dist(x,\partial\Omega)> \delta_2 R$, where $\delta_2<\frac12$ is a constant to fix in Remark \ref{remabove}. We denote this subcollection of cubes by  $\mathcal{W}_0$.

We say that $Q$ is \emph{peripheral} if it is not central.
 \end{definition}

\begin{rem}\label{remabove}
Consider $\delta_0<\frac12$ to be fixed. We call  ${\delta_0} \mathcal{Q}_k \cap \Omega$ the {\em canvas} of the window $\mathcal{Q}$, and we divide the peripheral cubes in collections $\mathcal{W}_k=\{Q\in \mathcal{W} \setminus \mathcal{W}_0 : Q \subset {\delta_0} \mathcal{Q}_k \cap \Omega\}$. For Whitney constants big enough and for $\delta_0$,  $\delta_1$ and $\delta_2$ small enough we have that 
\begin{enumerate}[1)]
\item The union of central cubes is a connected set.
\item Every peripheral cube is contained in a window canvas. The subcollections $\mathcal{W}_k$ are not disjoint and, if two peripheral cubes $Q$ and $S$ are not contained in any common $\mathcal{W}_k$, then $\dist(Q,S)\approx R$.
\item For each peripheral cube $Q\in\mathcal{W}_k$ there exists a cube $S\subset  \mathcal{Q}_k$ above $Q$ which is central.
\end{enumerate}
Furthermore, 
\begin{enumerate}[1)]
   \setcounter{enumi}{3}
\item All the central cubes have comparable side-length.
\end{enumerate}
\end{rem}

 Next we provide a tree-like structure to the family of cubes. 
\begin{definition}\label{defgeneralchain}
We say that $C=(Q_1,Q_2,\cdots, Q_{M})$ is a \emph{chain} connecting $Q_1$ and $Q_M$ if $Q_i$ and $Q_{i+1}$ are neighbors for every $i<M$. We will call the \emph{next} cube to $\mathcal{N}_{C}(Q_{i})=Q_{i+1}$. In general, we consider the iteration $\mathcal{N}_{C}^j(Q_{i})=Q_{i+j}$ whenever $i+j\leq M$. 
\end{definition}

We want to have a somewhat rigid structure to gain some control on the chains we use, so we need to introduce a \emph{chain function} $[\cdot,\cdot]:\mathcal{W}\times \mathcal{W}\to \bigcup_M\mathcal{W}^M$. We state three rules. The first one is on the definition of chain function.

\vspace{3mm}
\textbf{First rule:}
\begin{enumerate}[{1.}1:]
\item For any cubes $Q,S\in\mathcal{W}$, $[Q,S]$ is a chain connecting $Q$ and $S$.
\end{enumerate}

Abusing notation we will also write $[Q,S]$ for the non-ordered collection $\{Q_i\}_{i=1}^{M}$ so that we can say that $Q_i \in [Q,S]$. 
  
  Given two cubes $Q, S$, we will use the open-close interval notation $(Q,S):=[Q,S]\setminus \{Q,S\}$,  $[Q,S):=[Q,S]\setminus \{S\}$,  $(Q,S]:=[Q,S]\setminus \{Q\}$.

\vspace{6 em}

Now we can state the second rule, concerning the central cubes. For that purpose, assume that we have fixed a central cube $Q_0$.

\vspace{3mm}
\textbf{Second rule:}
\begin{enumerate}[2.1]
\item For every central cube $Q\in \mathcal{W}_0$, $[Q, Q_0]$ is a chain of central cubes connecting these two cubes with minimal number of steps.

\item For any central cubes $Q,S\in \mathcal{W}_0$ with $S\in [Q,Q_0]$, we have $[S,Q_0]\subset [Q,Q_0]$. Thus, we can define $[Q,S]=[Q,Q_0]\setminus(S,Q_0]$ (see Figure \ref{fig2.2}).

\begin{figure}[ht]
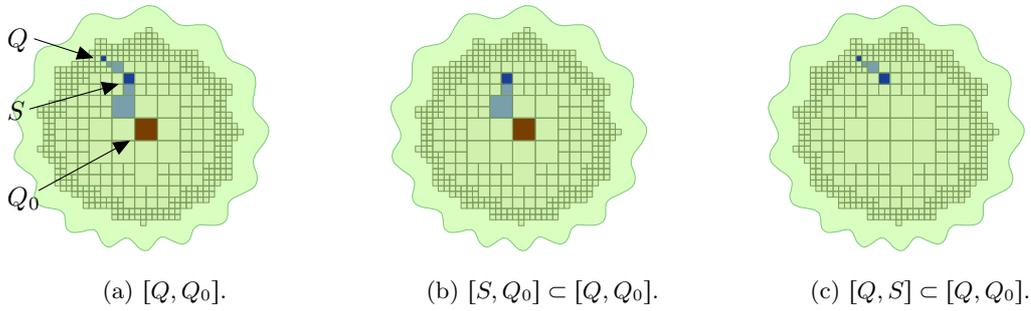

\caption{Second rule, 2.2.}\label{fig2.2}
\centering 
\begin{subfigure}[b]{0.3\textwidth}

\caption{$[Q,Q_0]$.}
\end{subfigure}
\quad 
\begin{subfigure}[b]{0.3\textwidth}
%
\caption{$[S,Q_0]\subset[Q,Q_0]$.}
\end{subfigure}
\quad 
\begin{subfigure}[b]{0.3\textwidth}
%
\caption{$[Q,S]\subset [Q,Q_0]$.}
\end{subfigure}
\end{figure}

\item Given two different central cubes $Q$ and $S$, let $Q_S$  be the first cube in $[Q,Q_0]$ with a neighbor in $[S,Q_0]$ and let $\tilde{S}_Q$ be the first neighbor of $Q_S$ in  $[S,Q_0]$. Then, $[Q,S]=[Q,Q_S]\cup [\tilde{S}_Q,S]$ (see Figure \ref{fig2.3}).

\end{enumerate}

\begin{figure}[ht]
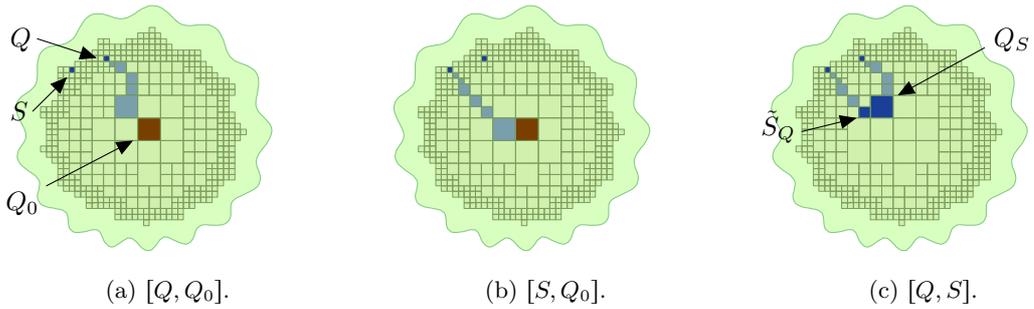

\caption{Second rule, 2.3.}\label{fig2.3}
\centering 

\begin{subfigure}[b]{0.3\textwidth}
%
\caption{$[Q,Q_0]$.}
\end{subfigure}
\quad 
\begin{subfigure}[b]{0.3\textwidth}
%
\caption{$[S,Q_0]$.}
\end{subfigure}
\quad 
\begin{subfigure}[b]{0.3\textwidth}
%
\caption{$[Q,S]$.}
\end{subfigure}
\end{figure}

Note that $S_Q$ may be different from $\tilde{S}_Q$. Abusing notation we will always write $S_Q$. This completes the central structure. For every Whitney cube $Q\subset \delta_0\mathcal{Q}_k$, we define $[Q,Q_0]_k$ as a chain connecting $Q$ and $Q_0$ and such that each cube $S\in [Q,Q_0]_k$ is either central or above $Q$ with respect to $\mathcal{Q}_k$, and in case $S$ is central, then $[Q,Q_0]_k=[Q,S]_k\cup[S,Q_0]$, where $[Q,S]_k$ is the subchain of $[Q,Q_0]_k$ limited by $Q$ and $S$ (see Figure \ref{figqk}). The chain $[Q,Q_0]_k$ exists in virtue of Remark \ref{remabove}.

\vspace{9 em}
 
\begin{figure}[ht]
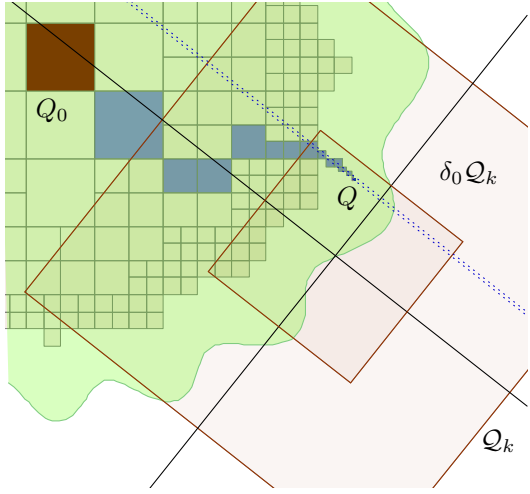

\center

\caption{$[Q,Q_0]_k$ for $Q\subset {\delta_0}\mathcal{Q}_k$.}\label{figqk}
\end{figure}
 Now we can add the rule for peripheral cubes.

\vspace{3mm}
\textbf{Third rule:}
 \begin{enumerate}[3.1:]
\item \label{itemrulechain} Given two diferent peripheral cubes which are both contained in, at least, one common window canvas $Q,S\in \mathcal{W}_k$, fix $k$ and use $[,]_k$: Define $Q_S\in [Q,Q_0]_k$, $S_Q\in [S,Q_0]_k$ and $[Q,S]=[Q,Q_S]_k\cup [S_Q,S]_k$ as in rule 2.3.
\item For every peripheral
 cube $S$, fix any $k$ such that $S \in \mathcal{W}_k$ and define $[S,Q_0]:=[S,Q_0]_k$.
\item Given two diferent cubes $Q$ and $S$ in any situation different from 3.1, use rule 2.3.
\end{enumerate}

\begin{definition}\label{defOrientedWhitney}
Given a Lipschitz domain $\Omega$, we say that $\{\mathcal{W}, \{\mathcal{Q}_k\}_{k=1}^N, Q_0, [\cdot,\cdot]\}$ is an \emph{oriented Whitney covering} of  $\Omega$ if $\mathcal{W}$ is a Whitney covering of $\Omega$ (see Definition \ref{defwhitney}), $\mathcal{Q}_k$ are windows satisfying \rf{eqwindowscontain}, the cube $Q_0 \in \mathcal{W}$ is a central cube of $\Omega$ with respect to those windows and $[\cdot,\cdot]$ is a chain function satisfying the three rules explained before. All the constants are fixed in Remark \ref{remabove}.

We say that the covering is \emph{properly oriented} with respect to a window $\mathcal{Q}_k$ if the cubes in the Whitney covering have sides parallel to the faces of $\mathcal{Q}_k$.\end{definition}

\begin{definition}
If $Q, S\in [P,Q_0]$ for some $P$ and $\mathcal{N}^j_{[P,Q_0]}(Q)=S$ for some $j\geq 0$, then we say that ${Q\leq S}$. We will say that $Q<S$ if $Q\leq S$ and $Q\neq S$. 
\end{definition}
\begin{rem}\label{remvertical}
If the covering is properly oriented with respect to $\mathcal{Q}_k$ and $Q, S \in \mathcal{W}_k$, then $Q\leq S$ if and only if $S\in[Q,Q_0]$. Otherwise, $Q\leq S$ does not imply that $S\in [Q,Q_0]$, but if $Q$ and $S$ are peripheral it implies that their vertical projections in some window have non-empty intersection.
\end{rem}

\begin{definition}
Given two cubes $Q$ and $S$ of an oriented Whitney covering, we define the \emph{long distance}
$$\Dist(Q,S)=\ell(Q)+\ell(S)+\dist(Q,S).$$
\end{definition}

\begin{rem}\label{remDist}
Using the properties of the Whitney covering, Remark \ref{remabove} and the chain function rules 2.3, 3.1 and 3.3, one can prove that,
 for $P \in [Q,Q_S]$, 
$$\Dist(P,S)\approx \Dist(Q,S)$$
and
$$\Dist(P,Q)\approx \ell(P).$$

\end{rem}

Now we consider the \emph{Hardy-Littlewood maximal operator},
$$Mg(x)= \sup_{Q \owns x} \fint_Q g(y) dy.$$
It is a well known fact that this operator is bounded in $L^p$ for $1<p\leq \infty$.
\begin{lemma}\label{lemmaximal}
Assume that $g\in L^1_{loc}$ and $r>0$. For every $Q\in\mathcal{W}$, we have
\begin{enumerate}[1)]
\item If $\eta>0$, 
	$$ \sum_{S:\Dist(Q,S)>r}  \frac{\int_S g(x) \, dx}{D(Q,S)^{d+\eta}}\lesssim \frac{\inf_{y\in Q} Mg(y)}{r ^\eta}.$$
\item If $\eta>0$, 
	$$ \sum_{S:\Dist(Q,S)<r}  \frac{\int_S g(x) \, dx}{D(Q,S)^{d-\eta}}\lesssim \inf_{y\in Q} Mg(y) \,r^\eta.$$
\item In particular, 
	$$ \sum_{S: S<Q} \int_S g(x) \, dx\lesssim \inf_{y\in Q} Mg(y) \, \ell(Q)^d.$$
\end{enumerate}
\end{lemma}

\begin{proof}
The sum in 1) can  just bounded by
	$$ C \int \frac{ g(x) \, dx}{(|x-y|+r)^{d+\eta}}$$
for every $y\in Q$, and this can be bounded separating the integral region in dyadic annuli.
The sum in (2) can be bounded by an analogous reasoning. Using the property W\ref{itemWvertical} of Definition \ref{defwhitney} we can see that 3)
is a particular case of 2) for $\eta=d$.
\end{proof}
Note that we used the Lipschitz character of $\Omega$ only to prove 3). In Section \ref{secmain2} we  will make use of the following technical results, specific for Lipschitz domains, which sharpen the results of the previous lemma for $g$ constant.
\begin{lemma}\label{lemad-1}
Let $a>d-1$ and $Q$ a Whitney cube. Then
$$\sum_{S\leq Q}\ell(S)^{a}\approx \ell(Q)^a$$
with constants depending only on $a$ and $d$.
\end{lemma}
\begin{proof}
First assume that $Q$ is not central. Selecting the cubes by their side-length, we can write
\begin{align*}
\sum_{S<Q}\ell(S)^{a}
	&=\sum_{j=1}^\infty\sum_{\substack{S<Q\\ \ell(S)=2^{-j}\ell(Q)}}(2^{-j}\ell(Q))^{a}
	=\ell(Q)^{a}\sum_{j=1}^\infty2^{-ja} \# \{S<Q: \ell(S)=2^{-j}\ell(Q)\}.
\end{align*}
Using W\ref{itemWvertical} and Remark \ref{remvertical} we get that 
$$ \# \{S<Q: \ell(S)=2^{-j}\ell(Q)\}\leq C2^{(d-1)j}$$
and thus
$$\sum_{S<Q}\ell(S)^{a}\lesssim \ell(Q)^a\sum_{j=1}^\infty2^{-j(a-(d-1))}.$$
This is bounded if $a>d-1$.

By the same token, given an $R$-window $\mathcal{Q}_k$,
\begin{equation*}
\sum_{S\subset\mathcal{Q}_k}\ell(S)^{a}\lesssim R^a.
\end{equation*}
Thus, the lemma is also valid for $Q$ central by the last statement of Remark \ref{remabove}.
\end{proof}

\begin{lemma}\label{lembad-1}
Let $b>a>d-1$ and $Q$ a Whitney cube. Then
$$\sum_{S\in\mathcal{W}}\frac{\ell(S)^a}{\Dist(Q,S)^b}\leq C \ell(Q)^{a-b},$$
with $C$ depending only on $a$, $b$ and $d$.
\end{lemma}
\begin{proof}

Let us assume that $Q\in \mathcal{W}_k$.
First of all we consider the cubes contained in  $\mathcal{Q}_k$ and we classify those cubes by their side-length and their distance to $Q$:
\begin{align*}
\sum_{S\subset\mathcal{Q}_k}\frac{\ell(S)^a}{\Dist(Q,S)^b}
	& \leq \sum_{i=-\infty}^\infty \sum_{j=0}^\infty \sum_{\substack{S: \ell(S)=2^i\ell(Q)\\ 2^j \ell(Q)\leq \Dist(S,Q)<2^{j+1}\ell(Q)}} \frac{(2^i \ell(Q))^a}{(2^j \ell(Q))^b} \\
	& \leq \ell(Q)^{a-b} \sum_{i=-\infty}^\infty \sum_{j=0}^\infty \frac{2^{ia}}{2^{jb}} \#\{S: \ell(S)=2^i\ell(Q) \mbox{, } \Dist(S,Q)<2^{j+1}\ell(Q)\} .
\end{align*}
Note that the value of $j$ in the last sum must be greater or equal than $i$ because, otherwise, the last cardinal would be zero. 

Using again W\ref{itemWvertical}, we can see that
\begin{align*}
\#\{S\in \mathcal{W}_k: \ell(S)=2^i\ell(Q) \mbox{, } \Dist (S,Q)<2^{j+1}\ell(Q)\}
	& \leq C \left(\frac{(2^{j+1})\ell(Q)}{2^i\ell(Q)}\right)^{d-1} = C 2^{(j-i)(d-1)}.
\end{align*}
Thus, 
\begin{align*}
\sum_{S\subset\mathcal{Q}_k}\frac{\ell(S)^a}{\Dist(Q,S)^b}
	& \lesssim \ell(Q)^{a-b} \sum_{j=0}^\infty \sum_{i=-\infty}^j 2^{i(a+1-d)-j(b+1-d)} \leq C_{a,b,d} \ell(Q)^{a-b}
\end{align*}
as soon as $b>a>d-1$.

On the other hand, when $S\nsubset  \mathcal{Q}_k$ the long distance $\Dist (Q,S)$ is always bounded from below by a constant times $R$ (because $Q \subset \delta_0\mathcal{Q}_k$), so separating $\mathcal{W}$ in subcollections $\mathcal{W}_k$ and using Lemma \ref{lemad-1}, 
\begin{align}\label{eqacotalallunyania}
\sum_{S\nsubset \mathcal{Q}_k}\frac{\ell(S)^a}{\Dist(Q,S)^b}
	& \lesssim \sum_{S \in\mathcal{W}_0} \frac{(\diam\Omega)^a}{R^b} + \sum_{j\neq k} \sum_{S \in\mathcal{W}_j} \frac{\ell(S)^a}{R^b} \lesssim R^{a-b}\lesssim \ell(Q)^{a-b}.
\end{align}

To prove the Lemma for a central cube $Q\in \mathcal{W}_0$, just apply an argument analogous to \rf{eqacotalallunyania}.
\end{proof}

\section{Approximating Polynomials}\label{secpoly}
Recall that the Poincar\'e inequality tells us that, given a cube $Q$ and a function $f\in W^{1,p}(Q)$ with $0$ mean in the cube, 
$$\norm{f}_{L^p(Q)}\lesssim \ell(Q) \norm{\nabla f}_{L^p(Q)}$$
with universal constants once we fix $d$ and $1\leq p< \infty$ (see, for example, \cite[Theorem 4.4.2]{ziemer}).

If we want to iterate that inequality, we also need the gradient of $f$ to have 0 mean on $Q$. That leads us to define the next approximating polynomials.
\begin{definition}
Let $\Omega$ be a domain and a cube $Q \subset \Omega$. Given $f\in L^1(Q)$ with weak derivatives up to order $n$, we define $\mathbf{P}^{n}_Q (f)\in \mathcal{P}^{n}$ as the unique polynomial (restricted to $\Omega$) of degree smaller or equal than $n$ such that 
\begin{equation}\label{eqdefpnQ}
\fint_{Q} D^\beta \mathbf{P}_Q^n f \,dm=\fint_{Q} D^\beta f\, dm
\end{equation}
for every multiindex $\beta \in \N^d$ with $|\beta| \leq n$.
\end{definition}

Note that these polynomials can be understood as a particular case of the projection $L:W^{1,p}(Q) \to \mathcal{P}^n $ introduced by Norman G. Meyers in \cite{meyers}.

\begin{lemma}\label{lempoly}
Given a cube $Q$ and $f\in W^{n-1,1}(3Q)$, the polynomial $\mathbf{P}^{n-1}_{3Q} f\in \mathcal{P}^{n-1}$ exists and is unique.
Furthermore, this polynomial has the next properties:
\begin{enumerate}[P1.]
\item\label{itemMGammaAcotat} Let $x_Q$ be the center of $Q$. If we consider the Taylor expansion of $\mathbf{P}_{3Q}^{n-1} f$ at $x_Q$, 
\begin{equation}\label{eqtaylorexp}
\mathbf{P}_{3Q}^{n-1} f (y)= \sum_{\substack{\gamma\in\N^d \\ |\gamma|<n}}m_{Q,\gamma} (y-x_Q)^\gamma,
\end{equation}
then the coefficients $m_{Q,\gamma}$ are bounded by
$$|m_{Q,\gamma}| \leq c_n  \sum_{j=|\gamma|}^{n-1} \norm{\nabla^{j} f}_{L^\infty(3Q)} \ell(Q)^{j-|\gamma|}. $$

\item\label{itemPpoincare}Furthermore, if $f\in W^{n,p}(3Q)$, for $1\leq p<\infty$ we have
 $$\|f-\mathbf{P}_{3Q}^{n-1} f\|_{L^p(3Q)}\leq C \ell(Q)^n \norm{\nabla^n f}_{L^p(3Q)} .$$

\item\label{itemPchain} Given an oriented Whitney covering $\mathcal{W}$ with chain function $[\cdot,\cdot]$ associated to $\Omega$, and given two Whitney cubes $Q, S\in \mathcal{W}$ and $f\in W^{n,p}(\Omega)$, 
\begin{equation*}
\norm{f-\mathbf{P}^{n-1}_{3Q} f}_{L^1(S)}  \leq \sum_{P\in [S,Q]}\frac{\ell(S)^d D(P,S)^{n-1}}{\ell(P)^{d-1}}\norm{\nabla^n f}_{L^1(3P)}.
\end{equation*}
%
%
\end{enumerate}
\end{lemma}

\begin{proof}
Note that \rf{eqdefpnQ} is a triangular system of equations on the coefficients of the polynomial. 
Indeed, for $\gamma$ fixed, if the polynomial exists and has Taylor expansion \rf{eqtaylorexp}, then
$$D^\gamma \mathbf{P}^{n-1}_{3Q} f (y)=\sum_{\beta\geq\gamma}m_{Q,\beta}\frac{\beta!}{(\beta-\gamma)!}(y-x_Q)^{\beta-\gamma}.$$
 When we take means on the cube $3Q$, 
\begin{align*}
\fint_{3Q} D^\gamma f \, dm
	&= \fint_{3Q} D^\gamma \mathbf{P}_{3Q}^{n-1} f  \, dm\\
	& =\sum_{\beta\geq\gamma}m_{Q,\beta}\frac{\beta!}{(\beta-\gamma)!} \left(\frac32\ell(Q)\right)^{|\beta-\gamma|}\fint_{Q(0,1)} y^{\beta-\gamma} dy \\
	& =\sum_{\beta\geq\gamma}C_{\beta,\gamma} m_{Q,\beta} \ell(Q)^{|\beta-\gamma|},
\end{align*}
which is a triangular system of equations on the coefficients $m_{Q, \beta}$. 

Solving for $m_{Q,\gamma}$, since $C_{\gamma, \gamma}\neq 0$ we obtain the explicit expression
\begin{equation}\label{eqmQgammaaillat}
m_{Q,\gamma} = \frac{1}{C_{\gamma,\gamma}} \fint_{3Q} D^\gamma f \, dm - \sum_{\beta>\gamma} C_{\beta,\gamma}m_{Q,\beta} \ell(Q)^{|\beta-\gamma|}.
\end{equation}
For $|\gamma|=n-1$ this gives the value of $m_{Q,\gamma}$ in terms of $D^\gamma f$, 
$$m_{Q,\gamma} = \frac{1}{C_{\gamma,\gamma}} \fint_{3Q} D^\gamma f \, dm. $$ 
Using induction on $n-|\gamma|$ we get the existence and uniqueness of $\mathbf{P}^{n-1}_{3Q} f$.
Taking absolute values we obtain P\ref{itemMGammaAcotat}.

The equality \rf{eqdefpnQ} allows us to iterate the Poincar\'e inequality 
 $$\|f-\mathbf{P}_{3Q}^{n-1} f\|_{L^p(3Q)}\leq C \ell(Q) \|\nabla (f-\mathbf{P}_{3Q}^{n-1} f) \|_{L^p(3Q)}\leq \dots\leq  C^n \ell(Q)^n \norm{\nabla^n f}_{L^p(3Q)},$$
that is, P\ref{itemPpoincare}.

To prove P\ref{itemPchain}, we consider the chain function in Definition \ref{defOrientedWhitney} to write
\begin{align}\label{eqCadenaF}
\norm{f-\mathbf{P}^{n-1}_{3Q} f}_{L^1(S)} \leq \norm{f-\mathbf{P}^{n-1}_{3S} f}_{L^1(S)} + \sum_{P\in [S,Q)}\norm{\mathbf{P}^{n-1}_{3P} f-\mathbf{P}^{n-1}_{3\mathcal{N}(P)} f}_{L^1(S)}\end{align}
where we write $\mathcal{N}(P)$ instead of $\mathcal{N}_{[S,Q]}(P)$ from Definition \ref{defgeneralchain}.
For every polynomial $q \in \mathcal{P}^{n-1}$, from the equivalence of norms of polynomials of bounded degree $\mathcal{P}^{n-1}$ it follows that
\begin{equation*}
\norm{q}_{L^1(Q)}\approx \ell(Q)^d \norm{q}_{L^\infty(Q)},
\end{equation*}
and for $r>1$, also
\begin{equation*}
\norm{q}_{L^\infty(rQ)}\lesssim r^{n-1} \norm{q}_{L^\infty(Q)},
\end{equation*}
with constants depending only on $d$ and $n$.
Applying these estimates to $q=\mathbf{P}^{n-1}_{3P} f-\mathbf{P}^{n-1}_{3\mathcal{N}(P)} f$ with $r\approx \frac{\Dist(P,S)}{\ell(P)}$, it follows that
\begin{align*}
\norm{\mathbf{P}^{n-1}_{3P} f-\mathbf{P}^{n-1}_{3\mathcal{N}(P)} f}_{L^1(S)}
	&\approx \norm{\mathbf{P}^{n-1}_{3P} f-\mathbf{P}^{n-1}_{3\mathcal{N}(P)} f}_{L^\infty(S)}\ell(S)^d\\
	&\lesssim \norm{\mathbf{P}^{n-1}_{3P} f-\mathbf{P}^{n-1}_{3\mathcal{N}(P)} f}_{L^\infty(3P\cap3\mathcal{N}(P))}\frac{\ell(S)^d D(P,S)^{n-1}}{\ell(P)^{n-1}}\\
	&\approx \norm{\mathbf{P}^{n-1}_{3P} f-\mathbf{P}^{n-1}_{3\mathcal{N}(P)} f}_{L^1(3P\cap3\mathcal{N}(P))}\frac{\ell(S)^d D(P,S)^{n-1}}{\ell(P)^{n-1}\ell(P)^d}.\\
\end{align*}
Using this estimate in \rf{eqCadenaF} and P\ref{itemPpoincare} we get
\begin{align*}
\norm{f-\mathbf{P}^{n-1}_{3Q} f}_{L^1(S)}
	&  \lesssim   \sum_{P\in [S,Q)}\left(\norm{\mathbf{P}^{n-1}_{3P} f-f}_{L^1(3P)}+\norm{f-\mathbf{P}^{n-1}_{3\mathcal{N}(P)} f}_{L^1(3\mathcal{N}(P))}\right)\frac{\ell(S)^d D(P,S)^{n-1}}{\ell(P)^{d+n-1}}\\
	&  \lesssim  \sum_{P\in [S,Q]}\norm{f-\mathbf{P}^{n-1}_{3P} f}_{L^1(3P)}\frac{\ell(S)^d D(P,S)^{n-1}}{\ell(P)^{d+n-1}}\\
	&  \leq  \sum_{P\in [S,Q]}\norm{\nabla^n f}_{L^1(3P)}\frac{\ell(S)^d D(P,S)^{n-1}}{\ell(P)^{d-1}}.
\end{align*}
%
%
\end{proof}

\section{Some remarks on the derivatives of $T_\Omega f$}\label{secderivatives}
From now on, we assume $T$ to be a smooth convolution Calder\'on-Zygmund operator of order $n$. Recall that for $f\in L^p$ and $x\notin \supp(f)$, 
$$Tf(x) = \int K(x-y)f(y) \, dy,$$
where the kernel $K$ has derivatives bounded by
\begin{equation}\label{eqCZKderivades}
|\nabla^jK(x)|\leq \frac{ C}{|x|^{d+j}} \mbox{\,\,\,\, for $0\leq j \leq n$}.
\end{equation}
Given a function $f\in W^{n,p}(\Omega)$, we want to see that its transform $T_\Omega f=\chi_\Omega\, T(\chi_\Omega \, f)$ is in some Sobolev space, so we need to check that its weak derivatives exist up to order $n$. Indeed that is the case.

\begin{lemma}\label{lemweaksense}
Given $f\in W^{n,p}(\Omega)$, the weak derivatives of $T_\Omega f$ in $\Omega$ exist up to order $n$.
\end{lemma}

Before proving this, we consider the functions defined in all $\R^d$.

\begin{rem}\label{remsobolevglobal}
Since $T$ is a bounded linear operator in $L^2(\R^d)$ that commutes with translations, for Schwartz functions the derivative commutes with $T$ (see \cite[Lemma 2.5.3]{grafakos}). Using that $\mathcal{S}$ is dense in $W^{n,p}$ (see \cite[sections 2.3.3 and 2.5.6]{triebel}, for instance), we conclude that for every $f \in W^{n,p}(\R^d)$
\begin{equation}\label{eqderivacommuta}
D^\alpha T(f)=TD^\alpha(f)
\end{equation} 
and, thus, the operator $T$ is bounded in $W^{n,p}(\R^d)$. 
\end{rem}

\begin{definition}
Let $K \in W^{n,1}_{loc}(\R^d \setminus \{0\})$ be the kernel of $T$ and consider a function $f\in L^p$, a multiindex $\alpha\in\N^d$ with $|\alpha|\leq n$ and $x\notin \supp(f)$. We define 
\begin{equation*}
T^{(\alpha)} f(x)=\int D^\alpha K (x-y) f(y)\,dy.
\end{equation*}
\end{definition} 

\begin{lemma}\label{lemderivaT}
Let  $f\in L^p$. Then $Tf$ has weak derivatives up to order $n$ in $\R^d\setminus\supp f$. Moreover, for every multiindex $\alpha \in \N^d$ with $|\alpha|\leq n$ and $x\notin \supp f$
$$D^\alpha T f(x)=T^{(\alpha)}f(x).$$
\end{lemma}

\begin{proof}
Take a compactly supported smooth function $\phi \in C^\infty_c(\R^d\setminus \supp f)$. We can use Fubini's Theorem and get
\begin{align*}
\langle T^{(\alpha)}f, \phi \rangle
	& =\int_{\supp \phi} \int_{\supp f} D^\alpha K (x-y) f(y) \, dy\, \phi(x) \,dx\\
	& = \int_{\supp f} \int_{\supp \phi}  D^\alpha K (x-y) \phi(x) \, dx\, f(y) \,dy.
\end{align*}
Using the definition of distributional derivative and Tonelli's Theorem again, 
\begin{align*}
\langle T^{(\alpha)}f, \phi \rangle
	& = (-1)^{|\alpha|} \int_{\supp f} \int_{\supp \phi}  K (x-y)  D^\alpha\phi(x)\, dx \, f(y)\, dy\\
	& = (-1)^{|\alpha|} \int_{\supp \phi} \int_{\supp f}   K (x-y)   f(y)\, dy \, D^\alpha\phi(x)\, dx = (-1)^\alpha \langle Tf, D^\alpha\phi \rangle.
\end{align*}
\end{proof}

\begin{proof}[Proof of Lemma \ref{lemweaksense}]
Take a classical Whitney covering of $\Omega$, $\mathcal{W}$, and for every $Q\in\mathcal{W}$, define a bump function $\varphi_Q \in C^\infty_c$ such that  $\chi_{2Q} \leq \varphi_Q\leq \chi_{3Q}$. On the other hand, let $\{\psi_Q\}_{Q\in  \mathcal{W}}$ be a partition of the unity associated to $\{\frac32 Q: Q\in \mathcal{W}\}$. 
Consider a multiindex $\alpha$ with $|\alpha|= n$.
Then take $f_1^Q=\varphi_Q \cdot f$, and $f_2^Q=(f-f_1^Q)\chi_\Omega$. One can define
$$g(y):=\sum_{Q\in  \mathcal{W}} \psi_Q(y) \left(TD^\alpha f_1^Q (y)+T^{(\alpha)} f_2^Q (y)\right).$$
 This function is defined almost everywhere in $\Omega$ and is the weak derivative $D^\alpha T_\Omega f$.

Indeed, given a test function $\phi \in C^\infty_c(\Omega)$, then, since $\phi$ is compactly supported in $\Omega$, its support intersects a finite number of Whitney double cubes and, thus,  the following additions are finite:
\begin{align}\label{eqtrencantgendos}
\nonumber \langle g ,\phi \rangle 
	& = \langle \sum_{Q\in  \mathcal{W}} \psi_Q \cdot TD^\alpha f_1^Q + \psi_Q \cdot  T^{(\alpha)} f_2^Q,\phi \rangle \\
	& =\sum_{Q\in  \mathcal{W}}\langle  TD^\alpha f_1^Q, \phi_Q \rangle
	+ \sum_{Q\in  \mathcal{W}}\langle  T^{(\alpha)} f_2^Q, \phi_Q \rangle, 
\end{align}
where $\phi_Q =\psi_Q \cdot \phi$. In the local part we can use \rf{eqderivacommuta}, so
\begin{equation*}
 \langle TD^\alpha f_1^Q, \phi_Q\rangle= (-1)^{|\alpha|} \langle  Tf_1^Q, D^\alpha (\phi_Q)\rangle.
\end{equation*}
When it comes to the non-local part, bearing in mind that $f_2^Q$ has support away form $2Q$ and $\phi_Q \in C^\infty_c(2Q)$, we can use the Lemma \ref{lemderivaT} and we get 
$$\langle T^{(\alpha)} f_2^Q, \phi_Q \rangle = (-1)^{|\alpha|}\langle Tf_2^Q, D^\alpha \phi_Q\rangle. $$

Back to \rf{eqtrencantgendos} we have
\begin{align*}
\langle g ,\phi \rangle
	&	=  \sum_{Q\in  \mathcal{W}} (-1)^{|\alpha|}\langle  Tf_1^Q, D^\alpha \phi_Q\rangle
	  	+ \sum_{Q\in  \mathcal{W}}(-1)^{|\alpha|}\langle Tf_2^Q, D^\alpha \phi_Q\rangle =  \sum_{Q\in  \mathcal{W}} (-1)^{|\alpha|}\langle  T_\Omega f, D^\alpha \phi_Q\rangle\\
	&	=   (-1)^{|\alpha|}\langle  T_\Omega f, D^\alpha \phi \rangle, 
\end{align*}
that is $g= D^\alpha T_\Omega f $ in the weak sense.
\end{proof}

\section{The Key Lemma}\label{seclemma}
To prove Theorem \ref{theoTP} we need the following lemma which says that it is equivalent to bound the transform of a function and its approximation by polynomials.

\begin{klemma}\label{lemklemma}
Let $\Omega$ be a Lipschitz domain, $\mathcal{W}$ an oriented Whitney covering associated to it (see Definition \ref{defOrientedWhitney}), $T$ a smooth convolution Calder\'on-Zygmund operator of order $n\in \N$ and $1<p<\infty$. Then the following statements are equivalent:
\begin{enumerate}[i)]
\item For every $f\in W^{n,p}(\Omega)$ one has
$$\norm{T_\Omega f}_{W^{n,p}(\Omega)}\leq C\norm{f}_{W^{n,p}(\Omega)},$$
where $C$ depends only on $n$, $p$, $T$ and the Lipschitz character of $\Omega$.
\item For every $f\in W^{n,p}(\Omega)$ one has
$$\sum_{Q\in\mathcal{W}} \norm{\nabla^n T_\Omega  (\mathbf{P}^{n-1}_{3Q} f)}_{L^p(Q)}^p\leq C\norm{f}^p_{W^{n,p}(\Omega)},$$
where $C$ depends only on $n$, $p$, $T$ and the Lipschitz character of $\Omega$.
\end{enumerate}
\end{klemma}

\begin{proof}
Given a multiindex $\alpha$ with $|\alpha|=n$, we will bound the difference 
\begin{equation}\label{eqcosaperacotar}
\sum_{Q\in\mathcal{W}} \norm {D^\alpha T_\Omega (f-\mathbf{P}^{n-1}_{3Q} f)}_{L^p(Q)}^p\lesssim \norm{\nabla^n f}^p_{L^p(\Omega)}.
\end{equation}

For each cube $Q\in \mathcal{W}$ we define a bump function $\varphi_Q \in C^\infty_c$ such that $\chi_{\frac32 Q}\leq \varphi_Q\leq \chi_{2Q}$ and $\norm{\nabla^j\varphi_Q}_\infty\approx \ell(Q)^{-j}$ for every $j\in\N$. Then we can break \rf{eqcosaperacotar} into local and non-local parts as follows:
\begin{align}\label{eqcosaperacotar12}
\nonumber \sum_{Q\in\mathcal{W}} \norm {D^\alpha T_\Omega (f-\mathbf{P}^{n-1}_{3Q} f)}_{L^p(Q)}^p
	& \lesssim \sum_{Q\in\mathcal{W}} \norm {D^\alpha T\left(\varphi_Q(f-\mathbf{P}^{n-1}_{3Q} f)\right)}_{L^p(Q)}^p \\
\nonumber	& \quad +\sum_{Q\in\mathcal{W}} \norm {D^\alpha T\left((\chi_\Omega-\varphi_Q)(f-\mathbf{P}^{n-1}_{3Q} f)\right)}_{L^p(Q)}^p\\
	& =\circled{1}+\circled{2}.
\end{align}

First of all we will show that the local term in \rf{eqcosaperacotar12} satisfies
\begin{equation}\label{eqcosaperacotar1}
\circled{1}=\sum_{Q\in\mathcal{W}} \norm {D^\alpha T\left(\varphi_Q(f-\mathbf{P}^{n-1}_{3Q} f)\right)}_{L^p(Q)}^p\lesssim \norm{\nabla^n f}^p_{L^p(\Omega)}.
\end{equation}
To do so, notice that $\varphi_Q(f-\mathbf{P}^{n-1}_{3Q} f)\in W^{n,p}(\R^d)$ and, by \rf{eqderivacommuta} and the boundedness of $T$ in $L^p$, 
\begin{align*}
\norm {D^\alpha T\left(\varphi_Q(f-\mathbf{P}^{n-1}_{3Q} f)\right)}_{L^p(Q)}^p 
	& \lesssim \norm{T}_{(p,p)}^p\norm{D^\alpha \left(\varphi_Q(f-\mathbf{P}^{n-1}_{3Q} f)\right)}_{L^p(\R^d)}^p \\
	& =C	\norm{D^\alpha \left(\varphi_Q(f-\mathbf{P}^{n-1}_{3Q} f)\right)}_{L^p(2Q)}^p ,
\end{align*}
where $\norm{\cdot}_{(p,p)}$ stands for the operator norm in $L^p (\R^d)$.
Using first the Leibnitz formula \rf{eqleibnitz}, and then using $j$ times the Poincar\'e inequality as in P\ref{itemPpoincare} from Lemma \ref{lempoly}, we get
\begin{align*}
\norm {D^\alpha T\left(\varphi_Q(f-\mathbf{P}^{n-1}_{3Q} f)\right)}_{L^p(Q)}^p 
	& \lesssim \sum_{j=1}^n \norm{\nabla^j\varphi_Q}_{L^\infty(2Q)}^p \norm{\nabla^{n-j}(f-\mathbf{P}^{n-1}_{3Q} f)}_{L^p(2Q)}^p \\
	& \lesssim \sum_{j=1}^n \frac1{\ell(Q)^{jp}} \ell(Q)^{jp} \norm{\nabla^n (f-\mathbf{P}^{n-1}_{3Q} f)}_{L^p(3Q)}^p = n \norm{\nabla^n f}_{L^p(3Q)}^p.
\end{align*}
Summing over all $Q$ we get \rf{eqcosaperacotar1}.

For the non-local part in  \rf{eqcosaperacotar12}, 
\begin{equation*}
\circled{2}=\sum_{Q\in\mathcal{W}} \norm {D^\alpha T\left((\chi_\Omega-\varphi_Q)(f-\mathbf{P}^{n-1}_{3Q} f)\right)}_{L^p(Q)}^p,
\end{equation*}
we will argue by duality. We can write
\begin{equation}\label{eqcosaperacotar21p}
\circled{2}^\frac1p=\sup_{\norm{g}_{L{p'}}\leq 1}\sum_{Q\in\mathcal{W}} \int_Q \left|D^\alpha T\left[(\chi_\Omega-\varphi_Q)(f-\mathbf{P}^{n-1}_{3Q} f)\right](x)\right| g(x) \, dx.
\end{equation}
Note that given $x\in Q$, by Lemma \ref{lemderivaT} one has
\begin{align*}
D^\alpha T[(\chi_\Omega-\varphi_Q) 
	& (f-\mathbf{P}^{n-1}_{3Q} f)] (x)  = \int_{\Omega} D^\alpha K (x-y) \left(1-\varphi_Q(y)\right)\left(f(y)-\mathbf{P}^{n-1}_{3Q} f(y)\right)dy.
\end{align*}
Taking absolute values and using Definition \ref{defCZK}, we can bound 
\begin{align}\label{eqvalorabsolut}
\nonumber |D^\alpha T[(\chi_\Omega-\varphi_Q) (f-\mathbf{P}^{n-1}_{3Q} f)] (x)|
 & \leq C_K \int_{\Omega\setminus \frac32 Q} \frac{|f(y)-\mathbf{P}^{n-1}_{3Q} f(y)|}{|x-y|^{n+d}}dy  \\
 & \lesssim \sum_{S\in\mathcal{W}} \frac{\norm{f-\mathbf{P}^{n-1}_{3Q} f}_{L^1(S)}}{D(Q,S)^{n+d}}.
\end{align}

By property P\ref{itemPchain} in Lemma \ref{lempoly} we have 
\begin{equation*}
\norm{f-\mathbf{P}^{n-1}_{3Q} f}_{L^1(S)}  \leq \sum_{P\in [S,Q]}\frac{\ell(S)^d D(P,S)^{n-1}}{\ell(P)^{d-1}}\norm{\nabla^n f}_{L^1(3P)},
\end{equation*}
so plugging this expression and \rf{eqvalorabsolut} into \rf{eqcosaperacotar21p}, we get
\begin{align*}
\circled{2}^\frac1p
	&\lesssim \sup_{\norm{g}_{{p'}}\leq 1}
\sum_{Q\in\mathcal{W}} \int_Q g(x) \, dx \sum_{S\in\mathcal{W}}\sum_{P\in [S,Q]}\frac{\ell(S)^d D(P,S)^{n-1}\norm{\nabla^n f}_{L^1(3P)}}{\ell(P)^{d-1} D(Q,S)^{n+d}}.
\end{align*}
Finally, we use that $P\in [S,Q]$ implies $\Dist(P,S)\lesssim \Dist(Q,S)$ (see Remark \ref{remDist}) to get
\begin{align*}
\circled{2}^\frac1p
	&\lesssim \sup_{\norm{g}_{{p'}}\leq 1}
\sum_{Q, S\in\mathcal{W}} \sum_{P\in [S,S_Q]} \int_Q g(x) \, dx \, \frac{\ell(S)^d \norm{\nabla^n f}_{L^1(3P)}}{\ell(P)^{d-1} D(Q,S)^{d+1}}\\
	& \quad + \sup_{\norm{g}_{{p'}}\leq 1}
\sum_{Q,S\in\mathcal{W}}\sum_{P\in [Q_S,Q]} \int_Q g(x) \, dx \, \frac{\ell(S)^d \norm{\nabla^n f}_{L^1(3P)}}{\ell(P)^{d-1} D(Q,S)^{d+1}}\\
	& = \circled{2.1}+\circled{2.2}.
\end{align*}

We consider first the term $\circled{2.1}$ where $P\in [S,S_Q]$ and, thus, by Remark \ref{remDist} the long distance $\Dist(Q,S)\approx \Dist(P,Q)$. Rearranging the sum, 
\begin{align*}
\circled{2.1}
	&\lesssim \sup_{\norm{g}_{{p'}}\leq 1}
 \sum_{P\in\mathcal{W}}  \frac{\norm{\nabla^n f}_{L^1(3P)}}{\ell(P)^{d-1}} \sum_{Q\in\mathcal{W}}  \frac{\int_Q g(x) \, dx}{D(Q,P)^{d+1}}  \sum_{S\leq P} \ell(S)^d.
\end{align*}
By Lemma \ref{lemmaximal}, 
$$\sum_{S\leq P} \ell(S)^d \approx \ell(P)^d,$$
and
$$ \sum_{Q\in\mathcal{W}}  \frac{\int_Q g(x) \, dx}{D(Q,P)^{d+1}}\lesssim \frac{\inf_{x\in3P} Mg(x)}{\ell(P)}.$$

Next we perform a similar argument with $\circled{2.2}$. Note that when $P\in[Q,Q_S]$, we have $\Dist(Q,S) \approx \Dist(P,S)$, leading to
\begin{align*}
\circled{2.2}
	&\lesssim \sup_{\norm{g}_{{p'}}\leq 1}
 \sum_{P\in\mathcal{W}}  \frac{\norm{\nabla^n f}_{L^1(3P)}}{\ell(P)^{d-1}} \sum_{Q\leq P}  \int_Q g(x) \, dx  \sum_{S} \frac{\ell(S)^d}{D(P,S)^{d+1}}.
\end{align*}
By Lemma \ref{lemmaximal}, 
$$\sum_{Q\leq P}  \int_Q g(x) \, dx \lesssim  \inf_{x\in3P} Mg(x)\, \ell(P)^d,$$
and
$$ \sum_{S}  \frac{\ell(S)^d}{D(P,S)^{d+1}}\approx \frac{1}{\ell(P)}.$$
Thus, 
\begin{align*}
\circled{2.1}+\circled{2.2}
	&\lesssim \sup_{\norm{g}_{{p'}}\leq 1}
 \sum_{P\in\mathcal{W}}  \frac{\norm{\nabla^n f}_{L^1(3P)}}{\ell(P)^{d-1}} \frac{\inf_{3P} Mg}{\ell(P)} \ell(P)^d \lesssim \sup_{\norm{g}_{{p'}}\leq 1}
 \sum_{P\in\mathcal{W}} \norm{\nabla^n f \cdot Mg}_{L^1(3P)}
\end{align*}
and, by H\"older inequality and the boundedness of the Hardy-Littlewood maximal operator in $L^{p'}$, 
\begin{align*}
\circled{2}^\frac{1}{p}
	&\lesssim \left(\sum_{P\in\mathcal{W}} \norm{\nabla^n f}_{L^p(3P)}^p\right)^{1/p} \left(\sup_{\norm{g}_{{p'}}\leq 1}
 \sum_P \norm{Mg}_{L^{p'}(3P)}^{p'}\right)^{1/p'} \lesssim \norm{\nabla^n f}_{L^p(\Omega)} .
\end{align*}
\end{proof}

\section{Proof of Theorem \ref{theoTP}}\label{secmain}

\begin{proof}
The implication $a) \Rightarrow b)$ is trivial.

To see the converse, fix a point $x_0 \in \Omega$. We have a finite number of monomials $P_\lambda(x)=(x-x_0)^\lambda$ for multiindices $\lambda\in\N^d$ and $|\lambda|<n$, so the hypothesis can be written as 
\begin{equation}\label{eqTP}
\|T_\Omega (P_\lambda)\|_{W^{n,p}(\Omega)}\leq C.
\end{equation} 

Assume $f\in W^{n,p}(\Omega)$. By the Key Lemma, we have to prove that 
$$\sum_{Q\in\mathcal{W}}\| \nabla^n T_\Omega (\mathbf{P}^{n-1}_{3Q} f)\|^p_{L^p(Q)}\lesssim \|f\|^p_{W^{n,p}(\Omega)}.$$
We can write the polynomials 
$$\mathbf{P}^{n-1}_{3Q} f(x)= \sum_{|\gamma|<n}m_{Q,\gamma}(x-x_Q)^\gamma ,$$
where $x_Q$ stands for the center of each cube $Q$. Taking the Taylor expansion in $x_0$ for each monomial, one has
$$\mathbf{P}^{n-1}_{3Q} f(x)= \sum_{|\gamma|<n}m_{Q,\gamma} \sum_{\vec{0}\leq\lambda\leq \gamma}\binom{\gamma}{\lambda} (x-x_0)^\lambda (x_0-x_Q)^{\gamma-\lambda}.$$
Thus, 
\begin{equation}\label{eqderivadadeT}
\nabla^n T_\Omega (\mathbf{P}^{n-1}_{3Q} f)(y)=\sum_{|\gamma|<n}m_{Q,\gamma} \sum_{\vec{0}\leq\lambda\leq \gamma}\binom{\gamma}{\lambda} (x_0-x_Q)^{\gamma-\lambda} \nabla^n(T_\Omega P_\lambda)(y).
\end{equation}
Recall the property P\ref{itemMGammaAcotat} in Lemma \ref{lempoly}, which states that 
\begin{equation}\label{eqacotacoeficients}
|m_{Q,\gamma}| \leq C  \sum_{j=|\gamma|}^{n-1} \norm{\nabla^{j} f}_{L^\infty(3Q)} \ell(Q)^{j-|\gamma|}\lesssim\sum_{j=|\gamma|}^{n-1} \norm{\nabla^{j} f}_{L^\infty(\Omega)} \diam \Omega^{j-|\gamma|}.
\end{equation}
Raising \rf{eqderivadadeT} to the power $p$, integrating in $Q$ and using \rf{eqacotacoeficients} we get
\begin{align*}
\norm{\nabla^n T_\Omega (\mathbf{P}^{n-1}_{3Q} f)}^p_{L^p(Q)}
	& \lesssim \sum_{j<n}\norm{\nabla^{j} f}_{L^\infty(\Omega)}^p \sum_{|\lambda|<j} \diam \Omega^{(j-|\lambda|)p}\norm{\nabla^n(T_\Omega P_\lambda)}_{L^p(Q)}^p.
\end{align*}

By the Sobolev Embedding Theorem, we know that $\norm{\nabla^j f}_{L^\infty(\Omega)}\leq C\norm{\nabla^j f}_{W^{1,p}(\Omega)}$ as long as $p>d$. If we add with respect to $Q\in \mathcal{W}$ and we use \rf{eqTP} we get
\begin{align*}
\sum_{Q\in\mathcal{W}}\norm{\nabla^n T_\Omega (\mathbf{P}^{n-1}_{3Q} f)}^p_{L^p(Q)}
	& \lesssim \sum_{j<n}\norm{\nabla^{j} f}_{W^{1,p}(\Omega)}^p \sum_{|\lambda|<j}\norm{\nabla^n(T_\Omega P_\lambda)}_{L^p(\Omega)}^p \lesssim \norm{f}_{W^{n,p}(\Omega)}^p ,
\end{align*}
with constants depending on the diameter of $\Omega$, $p$, $d$ and $n$. 
\end{proof}

\section{Carleson measures}\label{seccarleson}
Theorem \ref{theoTP} provides us with a nice tool to check if an operator is bounded in $W^{n,p}(\Omega)$ as long as $p>d$. Our concern for this section is to find a sufficient condition valid even if $p\leq d$. We want this condition to be related to some test functions (the polynomials of degree smaller than $n$ seem the right choice) but somewhat more specific than the condition in the Key Lemma. In particular we seek for some Carleson condition in the spirit of the celebrated article \cite{ars} by N. Arcozzi, R. Rochberg and E. Sawyer. In the next section we will check that, when we consider only the first derivative, that is for $W^{1,p}(\Omega)$, the sufficient condition below is in fact necessary. 

To use their techniques we need to have some tree structure coherent with the shadows of the cubes. We will use a local version of the Key Lemma in order to get rid of some technical difficulties:

\begin{lemma}\label{lemklemma2}
Let $\Omega\subset \R^d$ be a Lipschitz domain, $T$ a smooth convolution Calder\'on-Zygmund operator of order $n\in \N$ and $1<p<\infty$. Then the following statements are equivalent.
\begin{enumerate}[i)]
\item For every $f\in W^{n,p}(\Omega)$ one has
\begin{equation}\label{eqlema2.1}
\norm{T_\Omega f}_{W^{n,p}(\Omega)}\leq C\norm{f}_{W^{n,p}(\Omega)}.
\end{equation}
\item For every window $\mathcal{Q}$ and every $f\in W^{n,p}(\Omega)$ with $f|_{({\delta_0}\mathcal{Q})^c}\equiv 0$ one has
\begin{equation*}
\sum_{Q\in\mathcal{W}_\mathcal{Q}} \norm{\nabla^n T_\Omega  (\mathbf{P}^{n-1}_{3Q} f)}_{L^p(Q)}^p\leq C\norm{f}^p_{W^{n,p}(\Omega)},
\end{equation*}
where the whitney covering $\mathcal{W}_\mathcal{Q}$ is properly oriented with respect to $\mathcal{Q}$, that is, with the dyadic grid parallel to the local coordinates (see Definition \ref{defOrientedWhitney}). 
\end{enumerate}
\end{lemma}
\begin{proof}[Sketch of the proof]
To see that \textit{i)} implies \textit{ii)} just use the Key Lemma with an appropriate dyadic grid. 

To see the converse, one can choose a finite a collection of windows $\{\mathcal{Q}_k\}_{k=1}^N$ with $N\approx \mathcal{H}^{d-1}(\partial\Omega) / R^{d-1}$ such that $\frac{\delta_0}{c_0}\mathcal{Q}_k$ is a covering of the boundary of $\Omega$, call $\mathcal{Q}_0$ to the inner region $\Omega \setminus \bigcup \frac{{\delta_0}}{2}\mathcal{Q}_k$, and let $\{\psi_k\}\subset C^\infty$ be a partition of the unity related to the covering $ \{\mathcal{Q}_0\} \cup \{{\delta_0}\mathcal{Q}_k\}_{k=1}^N$.
Consider a function $f\in W^{n,p}(\Omega)$. Notice that our hypothesis does not give information about the inner region, but since $\psi_0$ is compactly supported in $\Omega$, $\psi_0 f\in W^{n,p}(\R^d)$ and by Remark \ref{remsobolevglobal} also $T(\psi_0 f)\in W^{n,p}(\R^d)$, so
$$\norm{T_\Omega(\psi_0 f)}_{W^{n,p}(\Omega)}=\norm{T(\psi_0 f)}_{W^{n,p}(\Omega)}\leq \norm{T(\psi_0 f)}_{W^{n,p}(\R^d)}\leq C\norm{\psi_0 f}_{W^{n,p}(\Omega)}.$$
Now, following the proof for the Key Lemma but replacing $f$ by $\psi_k f$ and using an appropriate Whitney covering for every single window, one gets 
$$\norm{T_\Omega (\psi_k f)}_{W^{n,p}(\Omega)}\leq C \norm{\psi_k f}_{W^{n,p}(\Omega)}.$$
Thus, 
$$\norm{T_\Omega f}_{W^{n,p}(\Omega)}\leq \sum_{k=0}^N \norm{T_\Omega (\psi_k f)}_{W^{n,p}(\Omega)} \leq C\sum_{k=0}^N \norm{\psi_k f}_{W^{n,p}(\Omega)}.$$
Choosing $\psi_k$ as bump functions with the usual estimates on the derivatives $\norm{\nabla^j \psi_k}_{L^\infty}\lesssim R^{-j}$, one can get \rf{eqlema2.1} using the Leibnitz formula.
\end{proof}

Next we recall some useful results from \cite{ars}. First we need to introduce some notation.

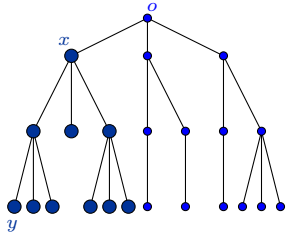
\begin{figure}[h]
\center
\begin{tikzpicture}[line cap=round,line join=round,>=triangle 45,x=0.5cm,y=0.5cm]
\clip(-3.01,0.28) rectangle (5.18,6.62);
\draw (1,6)-- (-1,5);
\draw (1,6)-- (1,5);
\draw (1,6)-- (3,5);
\draw (3,5)-- (4,3);
\draw (4,3)-- (4.5,1);
\draw (4,3)-- (4,1);
\draw (4,3)-- (3.5,1);
\draw (3,1)-- (3,3);
\draw (3,3)-- (3,5);
\draw (2,3)-- (1,5);
\draw (2,1)-- (2,3);
\draw (1,3)-- (1,5);
\draw (1,3)-- (1,1);
\draw (0.5,1)-- (0,3);
\draw (0,3)-- (0,1);
\draw (-0.5,1)-- (0,3);
\draw (0,3)-- (-1,5);
\draw (-1,3)-- (-1,5);
\draw (-1,5)-- (-2,3);
\draw (-2,3)-- (-2.5,1);
\draw (-2,3)-- (-2,1);
\draw (-2,3)-- (-1.5,1);
\begin{scriptsize}
\draw [fill=qqqqff] (1,6) circle (1.5pt);
\draw[color=qqqqff] (1.12,6.3) node {$o$};
\draw [fill=qqttzz] (-1,5) circle (2.5pt);
\draw[color=qqttzz] (-1.19,5.4) node {$x$};
\draw [fill=qqqqff] (1,5) circle (1.5pt);
\draw [fill=qqqqff] (3,5) circle (1.5pt);
\draw [fill=qqttzz] (-2,3) circle (2.5pt);
\draw [fill=qqttzz] (-1,3) circle (2.5pt);
\draw [fill=qqttzz] (0,3) circle (2.5pt);
\draw [fill=qqqqff] (1,3) circle (1.5pt);
\draw [fill=qqqqff] (2,3) circle (1.5pt);
\draw [fill=qqqqff] (3,3) circle (1.5pt);
\draw [fill=qqqqff] (4,3) circle (1.5pt);
\draw [fill=qqttzz] (-2.5,1) circle (2.5pt);
\draw[color=qqttzz] (-2.56,0.5) node {$y$};
\draw [fill=qqttzz] (-2,1) circle (2.5pt);
\draw [fill=qqttzz] (-1.5,1) circle (2.5pt);
\draw [fill=qqttzz] (-0.5,1) circle (2.5pt);
\draw [fill=qqttzz] (0,1) circle (2.5pt);
\draw [fill=qqttzz] (0.5,1) circle (2.5pt);
\draw [fill=qqqqff] (1,1) circle (1.5pt);
\draw [fill=qqqqff] (2,1) circle (1.5pt);
\draw [fill=qqqqff] (3,1) circle (1.5pt);
\draw [fill=qqqqff] (3.5,1) circle (1.5pt);
\draw [fill=qqqqff] (4,1) circle (1.5pt);
\draw [fill=qqqqff] (4.5,1) circle (1.5pt);
\end{scriptsize}
\end{tikzpicture}
\caption{$y \in \mathbf{Sh}_{\mathcal{T}}(x)$.}\label{figtree}
\end{figure}

\begin{definition}
We say that a connected, loopless graph $\mathcal{T}$ is a tree, and we will fix a vertex $o\in \mathcal{T}$ and call it its root. This choice induces a partial order in $\mathcal{T}$, given by $x \geq y$ if $x\in [o,y]$ where $[o,y]$ stands for the geodesic path uniting those two vertices of the graph (see Figure \ref{figtree}). We call shadow of $x$ in $\mathcal{T}$ to the collection
$${\mathbf{Sh}}_\mathcal{T}(x)=\{y \in \mathcal{T}: y\leq x\}.$$

We say that a function $\rho : \mathcal{T} \to \R$ is a weight if it takes positive values (by a function we mean a function defined in the vertices of the tree).
\end{definition}

\begin{rem}
Note that in \cite{ars} the notation is $\leq$ instead of $\geq$. We use the latter to be consistent with the tree structure of the Whitney covering that we introduce below.
\end{rem}

\begin{definition}
Given $h : \mathcal{T} \to \R$, we call the primitive $\mathcal{I}h$ the function
$$\mathcal{I}h (y)= \sum_{x\in [o, y]} h (x).$$
\end{definition}

\begin{theorem}\cite[Theorem 3]{ars}\label{theoars}
Let $1<p<\infty$ and let $\rho$ be a weight on $\mathcal{T}$. For a nonnegative measure $\mu$ on $\mathcal{T}$, the following statements are equivalent:
\begin{enumerate}[i)]
\item There exists a constant $C=C(\mu)$ such that
$$\norm{\mathcal{I}h}_{L^p(\mu)}\leq C \norm{h}_{L^p(\rho)}$$
\item There exists a constant $C=C(\mu)$ such that for every $r\in \mathcal{T}$ one has
$$\sum_{x\in {\mathbf{Sh}}_\mathcal{T}(r)} \left( \sum_{y\in {\mathbf{Sh}}_\mathcal{T}(x)}\mu(y)\right)^{p'}\rho(x)^{1-p'}\leq C \sum_{x\in {\mathbf{Sh}}_\mathcal{T}(r)}\mu(x).$$ 
\end{enumerate}
\end{theorem}

For every $1\leq p\leq\infty$, we say that a non-negative measure $\mu$ is a $p$-Carleson measure for $(\mathcal{I}, \rho, p)$ if there exists a constant $C=C(\mu)$ such that the condition \textit{i)} is satisfied.

Given an $R$-window $\mathcal{Q}$ of a Lipschitz domain $\Omega$ with a properly oriented Whitney covering $\mathcal{W}$, for every $x \in  \mathcal{Q}$, we write $x=(x', x_d)\in \R^{d-1}\times\R$ and, if $x$ is contained in  a Whitney cube $Q \in \mathcal{W}$, we define the shadow of $x$ as
$${\mathbf{Sh}}(x)=\left\{y\in \mathcal{Q}\cap \Omega : \, y_d<x_d \mbox{ and } \norm{x'-y'}_\infty\leq \frac12\ell(Q)\right\}.$$
Note that if $x$ is the center of the upper $(n-1)$-dimensional face of $Q$, the vertical projection of ${\mathbf{Sh}}(x)$ (which is a $(n-1)$-dimensional square) coincides with the vertical projection of $Q$ (see Figure \ref{figshadowQ}). Finally, we define the vertical extension of ${\mathbf{Sh}}(x)$, 
$$\widetilde {\mathbf{Sh}}(x)=\left\{y\in \mathcal{Q}\cap \Omega  : \, y_d< x_d+2\ell(Q) \mbox{ and } \norm{x'-y'}_\infty \leq\frac12 \ell(Q) \right\}.$$

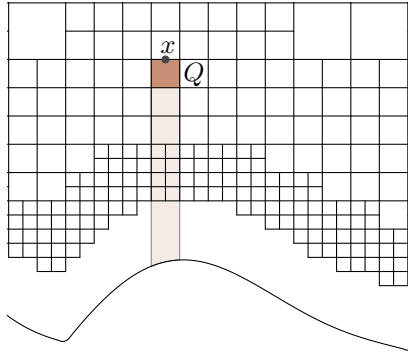
\begin{figure}[hb]
\center
\begin{tikzpicture}[line cap=round,line join=round,>=triangle 45,x=0.75cm,y=0.75cm]
\clip(-1.02,-0.51) rectangle (6,6);
\draw[help lines,fill=zzttqq,fill opacity=0.1] {[smooth,samples=50,domain=1.5:2.0] plot(\x,{sin((\x)*180/pi)/2+abs((\x-3)^2-9)/8})} -- (2,5) {[smooth,samples=50,domain=2.0:1.5] -- plot(\x,{5})} -- (1.5,1.34) -- cycle;
\draw[smooth,samples=100,domain=-1.6881871921415879:6.866443890546169] plot(\x,{sin(((\x))*180/pi)/2+abs(((\x)-3)^2-9)/8});
\fill[line width=0pt,color=zzttqq,fill=zzttqq,fill opacity=0.5] (1.5,4.5) -- (2,4.5) -- (2,5) -- (1.5,5) -- cycle;
\draw (1.9,5.1) node[anchor=north west] {$Q$};
\draw (-3,4.5)-- (-3,6);
\draw (-2.75,5.5)-- (-2.75,4);
\draw (-2.5,6)-- (-2.5,3.5);
\draw (-2.25,3)-- (-2.25,4.5);
\draw (-2,2.75)-- (-2,7);
\draw (-1.75,2.25)-- (-1.75,4);
\draw (-1.5,6)-- (-1.5,2);
\draw (-1,7)-- (-1,1.5);
\draw (-1.25,1.75)-- (-1.25,3);
\draw (-0.75,1.5)-- (-0.75,2.5);
\draw (-0.5,1.25)-- (-0.5,5);
\draw (-0.25,1.25)-- (-0.25,2.5);
\draw (0,1.25)-- (0,4);
\draw (0.25,1.5)-- (0.25,3);
\draw (0.5,1.75)-- (0.5,4);
\draw (0.75,2)-- (0.75,3.5);
\draw (1,2.25)-- (1,4);
\draw (1.25,2.25)-- (1.25,3.5);
\draw (1.5,2.5)-- (1.5,4);
\draw (1.75,2.5)-- (1.75,3.5);
\draw (2,2.5)-- (2,4.5);
\draw (2.25,2.5)-- (2.25,3.5);
\draw (2.5,2.5)-- (2.5,4.5);
\draw (2.75,2.5)-- (2.75,3.5);
\draw (3,2.25)-- (3,4.5);
\draw (3.25,2.25)-- (3.25,3.5);
\draw (3.5,2)-- (3.5,4.5);
\draw (4,1.75)-- (4,4.5);
\draw (4.25,1.5)-- (4.25,3);
\draw (4.5,1.5)-- (4.5,3.5);
\draw (4.75,1.25)-- (4.75,2.5);
\draw (5,1.25)-- (5,3.5);
\draw (5.25,1.25)-- (5.25,2.5);
\draw (5.5,1)-- (5.5,3.5);
\draw (5.75,1)-- (5.75,2);
\draw (6,1)-- (6,3.5);
\draw (6.25,1.25)-- (6.25,3);
\draw (6.5,1.75)-- (6.5,3.5);
\draw (6.75,2)-- (6.75,3.5);
\draw (7,2.25)-- (7,4.5);
\draw (7,4.5)-- (7,7);
\draw (6.5,6)-- (6.5,3.5);
\draw (6,3.5)-- (6,6);
\draw (5.5,5)-- (5.5,3.5);
\draw (5,6)-- (5,3.5);
\draw (5,6)-- (5,7);
\draw (6,6)-- (6,7);
\draw (5,6)-- (7,6);
\draw (5,7)-- (7,7);
\draw (7,5.5)-- (6,5.5);
\draw (7,5)-- (5,5);
\draw (7,4.5)-- (5,4.5);
\draw (5,4)-- (7,4);
\draw (7,3.5)-- (5,3.5);
\draw (7,3.25)-- (6.5,3.25);
\draw (5,3)-- (7,3);
\draw (7,2.75)-- (6,2.75);
\draw (7,2.5)-- (5,2.5);
\draw (7,2.25)-- (6,2.25);
\draw (6.75,2)-- (5,2);
\draw (5.5,2.25)-- (5,2.25);
\draw (6.5,1.75)-- (5,1.75);
\draw (6.25,1.5)-- (5,1.5);
\draw (6.25,1.25)-- (5,1.25);
\draw (6,1)-- (5.5,1);
\draw (5,1.25)-- (4.75,1.25);
\draw (5,1.5)-- (4.25,1.5);
\draw (5,1.75)-- (4,1.75);
\draw (5,2)-- (3.5,2);
\draw (5,2.25)-- (3,2.25);
\draw (5,2.5)-- (2.5,2.5);
\draw (4.5,2.75)-- (2.5,2.75);
\draw (4.5,3)-- (2.5,3);
\draw (3.5,3.25)-- (2.5,3.25);
\draw (4.5,3.5)-- (2.5,3.5);
\draw (5,3)-- (4.5,3);
\draw (5,3.5)-- (4.5,3.5);
\draw (5,4)-- (2.5,4);
\draw (5,4.5)-- (2.5,4.5);
\draw (5,5)-- (2.5,5);
\draw (4.5,3.5)-- (4.5,5);
\draw (4,5)-- (4,4.5);
\draw (4,5)-- (4,7);
\draw (5,7)-- (3,7);
\draw (2.5,6)-- (4,6);
\draw (4,5.5)-- (2.5,5.5);
\draw (3.5,4.5)-- (3.5,6);
\draw (3,4.5)-- (3,7);
\draw (2.5,4.5)-- (2.5,6);
\draw (3,7)-- (1,7);
\draw (2,7)-- (2,4.5);
\draw (2.5,6)-- (0.5,6);
\draw (2.5,5.5)-- (0.5,5.5);
\draw (2.5,5)-- (0.5,5);
\draw (2.5,4.5)-- (0.5,4.5);
\draw (2.5,4)-- (0.5,4);
\draw (2.5,3.5)-- (0.5,3.5);
\draw (2.5,3.25)-- (0.5,3.25);
\draw (1.5,4)-- (1.5,6);
\draw (1,7)-- (1,4);
\draw (3.75,3)-- (3.75,2);
\draw (2.5,2.5)-- (0,2.5);
\draw (2.5,2.75)-- (0,2.75);
\draw (2.5,3)-- (0,3);
\draw (0.5,3.5)-- (0,3.5);
\draw (0.5,4)-- (0,4);
\draw (1.25,2.25)-- (0,2.25);
\draw (0.75,2)-- (-1,2);
\draw (0.5,1.75)-- (-1,1.75);
\draw (0.25,1.5)-- (-1,1.5);
\draw (0,1.25)-- (-0.5,1.25);
\draw (0,2.25)-- (-1.25,2.25);
\draw (0,2.5)-- (-1.25,2.5);
\draw (0,3)-- (-1,3);
\draw (0,3.5)-- (-1,3.5);
\draw (0,4)-- (-1,4);
\draw (0.5,4)-- (0.5,6);
\draw (0,4)-- (0,7);
\draw (-1,6)-- (0.5,6);
\draw (0.5,5.5)-- (0,5.5);
\draw (0.5,5)-- (-1,5);
\draw (0.5,4.5)-- (-1,4.5);
\draw (-1,7)-- (1,7);
\draw (-1.25,1.75)-- (-1,1.75);
\draw (-1,2)-- (-1.5,2);
\draw (-1.25,2.25)-- (-1.75,2.25);
\draw (-1.25,2.5)-- (-1.75,2.5);
\draw (-1,2.75)-- (-2,2.75);
\draw (-1,3)-- (-2.25,3);
\draw (-1.5,3.25)-- (-2.25,3.25);
\draw (-1,3.5)-- (-2.5,3.5);
\draw (-1.5,3.75)-- (-2.5,3.75);
\draw (-1,4)-- (-2.75,4);
\draw (-2,4.25)-- (-2.75,4.25);
\draw (-1,4.5)-- (-3,4.5);
\draw (-3,4.75)-- (-2.5,4.75);
\draw (-1.5,5)-- (-3,5);
\draw (-2.5,5.25)-- (-3,5.25);
\draw (-1.5,5.5)-- (-3,5.5);
\draw (-1,6)-- (-3,6);
\draw (-1,5.5)-- (-1.5,5.5);
\draw (-1,5)-- (-1.5,5);
\draw (5,6)-- (4,6);
\draw (1.5,5.48) node[anchor=north west] {$x$};
\begin{scriptsize}
\fill [color=uuuuuu] (1.75,5) circle (1.5pt);
\end{scriptsize}
\end{tikzpicture}
\caption{The shadows $\mathbf{Sh}(x)$ and $\mathbf{Sh}(Q)$ coincide when $x$ is the center of the upper face of the cube. Furthermore, $P\subset \mathbf{Sh}(Q)$ if and only if $P\in \mathbf{Sh}_\mathcal{T}(Q)$.} \label{figshadowQ}
\end{figure}

More generally, given a set $U\subset \mathcal{Q}$ we call its shadow 
$${\mathbf{Sh}}(U)=\left\{y\in \mathcal{Q}\cap \Omega  : \mbox{ there exists } x \in U \mbox{ such that } y_d<x_d \mbox{ and } x'=y' \right\}.$$
Recall that we have a proper orientation in the Whitney covering. Thus, given a Whitney cube Q, we call the father of $Q$, $\mathcal{F}(Q)$ the neighbor Whitney cube which is immediately on top of $Q$ with respect to the vertical direction. This parental relation induces an order relation ($P\leq Q$ if $P$ is a descendant of $Q$). This would provide a tree structure to the Whitney covering $\mathcal{W}$ if there was a common ancestor $Q_0$ for all the cubes. This does not happen, but we can add a ``formal'' cube $Q_0$ (root of the tree) and then we can write $Q\leq Q_0$ for every $Q\subset \mathcal{Q}$. If we call $\mathcal{T}$ to the tree with the Whitney cubes as vertices complemented with $Q_0$ and the strucutre given by the order relation $\leq$, then  for every Whitney cube $Q\subset\mathcal{Q}$, 
$${\mathbf{Sh}}(Q)=\bigcup_{P\leq Q} P=\bigcup_{P\in \mathbf{Sh}_\mathcal{T}(Q)} P$$
(see Figure \ref{figshadowQ}). Since we will only consider functions and measures supported in the window canvas $\delta_0 \mathcal{Q}\cap \Omega$, we can extend any of them formally in $Q_0$ as the null function.

Now, some minor modifications in the proof of \cite[Proposition 16]{ars} allow us to rewrite this theorem in the following way.
\begin{propo}\label{propocarlesoncontinu}
Given $1<p<\infty$ and an $R$-window $\mathcal{Q}$ of a Lipschitz domain $\Omega$ with a properly oriented Whitney covering $\mathcal{W}$, consider the weights $\rho(x)=\dist(x, \partial\Omega)^{d-p}$, $\rho_\mathcal{W}(Q)=\ell(Q)^{d-p}$. For a positive Borel measure $\mu$ supported on $\delta_0 \mathcal{Q}\cap \Omega$, the following are equivalent:
\begin{enumerate} 
\item For every $a \in \delta_0 \mathcal{Q}\cap \Omega$ one has
$$\int_{\widetilde {\mathbf{Sh}}(a)} \rho(x)^{1-p'}(\mu({\mathbf{Sh}}(x)\cap {\mathbf{Sh}}(a)))^{p'} \frac{dx}{\dist(x, \partial\Omega)^d}\leq C \mu({\mathbf{Sh}}(a)).$$
\item For every $P\in \mathcal{W}$ one has
\begin{equation}\label{eqcarlesoncondition}
\sum_{Q\leq P} \left(\sum_{S\leq Q} \mu(S) \right)^{p'} \rho_\mathcal{W} (Q)^{1-p'}\leq C \sum_{Q\leq P} \mu(Q).
\end{equation}
\end{enumerate}
\end{propo}

In virtue of \cite[Theorem 1]{ars}, when $d=2$ and the domain $\Omega$ is the unit disk in the plane, the first condition is equivalent to $\mu$ being a Carleson measure for the analytic Besov space $B_p(\rho)$, that is, for every analytic function defined on the unit disc $\mathbb{D}$, 
$$\norm{f}_{L^p(\mu)}^p\lesssim \norm{f}_{B_p(\rho)}^p = |f(0)|^p+ \int_{\mathbb{D}} (1-|z|^2)^p |f'(z)|^p \rho(z) \frac{dm(z)}{(1-|z|^2)^2}.$$
\begin{definition}\label{defmesuraCarleson}
We say that a measure satisfying the hypothesis of Proposition \ref{propocarlesoncontinu} is a $p$-Carleson measure for $\mathcal{Q}$.

We say that a positive and finite Borel measure $\mu$ is a $p$-Carleson measure for a Lipschitz domain $\Omega$ if it is a $p$-Carleson measure for every $R$-window of the domain.
\end{definition}

We are ready to prove the second theorem. This proof is very much in the spirit of Theorem \ref{theoTP}. Again we fix a point $x_0 \in \Omega$ and we use the polynomials $P_\lambda(x)=(x-x_0)^\lambda$ for every multiindex $|\lambda|<n$, but now the key point is to use the Poincar\'e inequality instead of the Sobolev Embedding Theorem. Our hypothesis is reduced to $d\mu_\lambda(x)=|\nabla^n T_\Omega  P_\lambda (x)|^p dx$ being a $p$-Carleson measure for $\Omega$ for every $|\lambda|<n$.

\begin{proof}[\textbf{Proof of Theorem \ref{theocarleson}}]
Consider a fixed $R$-window $\mathcal{Q}$ and a properly oriented Whitney covering $\mathcal{W}$, that is, with dyadic grid parallel to the window faces. Making use of Lemma \ref{lemklemma2}, we only need to bound
\begin{equation*}
\sum_{Q\in\mathcal{W}} \norm{\nabla^n T_\Omega  (\mathbf{P}^{n-1}_{3Q} f)}_{L^p(Q)}^p\leq C\norm{f}^p_{W^{n,p}(\Omega)}
\end{equation*}
for every $f\in W^{n,p}(\Omega)$ with $f|_{({\delta_0}\mathcal{Q})^c}\equiv 0$.

Fix such a function $f$. Using the expression \rf{eqtaylorexp} and expanding it as in \rf{eqderivadadeT} at a fixed point $x_0 \in \Omega$, we have
\begin{align*}
\sum_{Q\in\mathcal{W}} \norm{\nabla^n T_\Omega  (\mathbf{P}^{n-1}_{3Q} f)}_{L^p(Q)}^p 
	& \lesssim \sum_{|\gamma|<n} \sum_{\vec{0}\leq \lambda \leq\gamma} C_{\gamma, \lambda, \Omega}\sum_{Q\in\mathcal{W}} |m_{Q,\gamma}|^p\norm{\nabla^n T_\Omega  P_\lambda}_{L^p(Q)}^p.
\end{align*}
Moreover, by induction on \rf{eqmQgammaaillat}, the coefficients are bounded by
\begin{align*}
|m_{Q,\gamma}|
	& \lesssim \sum_{|\beta|<n :\, \beta\geq\gamma} \ell(Q)^{|\beta-\gamma|} C_{\beta,\gamma}  \left|\fint_{3Q} D^\beta f \, dm\right| \lesssim \sum_{|\beta|<n :\, \beta\geq\gamma} C_{\beta,\gamma,R} \left|\fint_{3Q} D^\beta f \, dm\right|,
\end{align*}
so 
\begin{align*}
\sum_{Q\in\mathcal{W}} \norm{\nabla^n T_\Omega  (\mathbf{P}^{n-1}_{3Q} f)}_{L^p(Q)}^p 
	& \lesssim \sum_{\substack{|\beta|<n\\ \vec{0}\leq\lambda\leq \beta}}\sum_{Q\in\mathcal{W}}  \left|\fint_{3Q} D^\beta f \, dm\right|^p \mu_\lambda(Q).
\end{align*}

 Taking into account that $f|_{({\delta_0}\mathcal{Q})^c}\equiv 0$, we have $\fint_{3P} D^\beta f \, dm =0$ for $P$ close enough to the root $Q_0$. Thus, 
$$ \fint_{3Q} D^\beta f\, dm =\sum_{P\in[Q,Q_0)} \left(\fint_{3P} D^\beta f \, dm-\fint_{3{\mathcal{F}(P)}} D^\beta f\, dm\right),$$
and we can use the Poincar\'e inequality to find that
 \begin{align}\label{eqdespresdepoincare}
\sum_{Q\in\mathcal{W}} \norm{\nabla^n T_\Omega  (\mathbf{P}^{n-1}_{3Q} f)}_{L^p(Q)}^p 
	& \lesssim \sum_{\substack{|\beta|<n\\ \vec{0}\leq\lambda\leq \beta}}\sum_{Q\in\mathcal{W}} \left(\sum_{P\geq Q}    \ell(P) \fint_{5P} |\nabla  D^\beta f| \, dm \right)^p \mu_\lambda(Q) .
 \end{align}

By assumption, $\mu_\lambda$ is a $p$-Carleson measure for every $|\lambda|<n$, that is, it satisifies both conditions of Proposition \ref{propocarlesoncontinu}. By Theorem \ref{theoars}, we have that, for every $h \in l^p(\rho_{\mathcal{W}})$, 
\begin{equation}\label{eqacotaphi}
\sum_{Q\in\mathcal{W}} \left(\sum_{P \geq Q} h(P) \right)^p \mu_\lambda (Q) \leq C \sum_{Q\in\mathcal{W}} h(Q)^p \ell(Q)^{d-p},
\end{equation}
where $\rho_\mathcal{W}(Q)=\ell(Q)^{d-p}$.

Let us fix $\beta$ and $\lambda$ momentarily and take $h(P)=\ell(P)\fint_{5P} |\nabla D^\beta f| \, dm$ in \rf{eqacotaphi}. Using Jensen's inequality and the finite overlapping of the quintuple cubes, we have 
\begin{align}\label{eqacotanablabeta}
\sum_{Q\in\mathcal{W}} \left(\sum_{P\geq Q} \ell(P) \fint_{5P} |\nabla  D^\beta f|\, dm \right)^p \mu_\lambda(Q) 
\nonumber	& \leq C \sum_{Q\in\mathcal{W}} \left(\fint_{5Q} |\nabla D^\beta f|\, dm \right)^p \,\ell(Q)^{d} \\
\nonumber	& \lesssim  \sum_{Q\in\mathcal{W}} \fint_{5Q} |\nabla D^\beta f|^p \, dm \,\ell(Q)^{d} \\
	& \lesssim  \int_{\Omega} |\nabla D^\beta f|^p \, dm .
\end{align}

Plugging \rf{eqacotanablabeta} into \rf{eqdespresdepoincare} for each $\beta$ and $\lambda$, we get 
\begin{equation*}
\sum_{Q\in\mathcal{W}} \norm{\nabla^n T_\Omega  (\mathbf{P}^{n-1}_{3Q} f)}_{L^p(Q)}^p\leq C\norm{f}^p_{W^{n,p}(\Omega)}.
\end{equation*}
\end{proof}

%

\section{The remaining implication in Theorem \ref{theoTb}}\label{secmain2}
In this section we prove the implication $1.\implies 2.$ in Theorem \ref{theoTb}. First we need some tools from partial differential equations.

\begin{rem}
Given $g\in L^1_0(\overline{\R^d_+})$ and $d>2$, consider the function 
\begin{equation}\label{eqF}
F(x):= N[(R^{(d-1)}_d {g}) d\sigma](x)=\int_{\partial\R^d_+}\frac{(R^{(d-1)}_d {g})(y) }{(2-d)w_d|x-y|^{d-2}} d\sigma(y) \mbox{\quad \quad for }x\in \R^d_+, \end{equation}
where $N$ denotes the Newton potential \rf{eqdefpotencialnewton}, $R^{(d-1)}_d$ stands for the vertical component of the vectorial $(d-1)$-dimensional Riesz transform $R^{(d-1)}$ and $d\sigma$ is the hypersurface measure in $\partial\R^d_+$.  This function is well defined  since
\begin{align*}
\norm{R^{(d-1)}_d {g}}_{L^1(\sigma)}
	& \leq \int_{\partial \R^d_+}\int_{\R^d_+}\frac{z_d}{|y-z|^d} |{g}(z)|\, dz \, d\sigma(y)\\
	& = \int_{\R^d_+} \left(\int_{\partial \R^d_+} \frac{z_d}{|y-z|^d}  d\sigma(y) \right) |{g}(z)|\,dz \approx \norm{ {g}}_{1} 
\end{align*}
and, thus, the right-hand side of \rf{eqF} is an absolutely convergent integral  for each $x\in\R^d_+$, with $F(x)\leq \frac{\norm{{g}}_{1}}{|x_d|^{d-2}}$. By the same token, all the derivatives of $F$ are well defined, $F$ is $C^\infty(\R^d_+)$, harmonic and $\nabla F(x)=R^{(d-1)} [(R^{(d-1)}_d {g}) d\sigma](x)$. When $d=2$ we have to make the usual modifications.
\end{rem}

\begin{lemma}\label{lemNeumann}
Consider a ball $B_1\subset \R^d$ centered at the origin and a real number $\varepsilon>0$. Let $g\in L^1(\R^d_+\cap \frac14 B_1)$ with $g(x',x_d)=0$ for every $(x',x_d)\in \R^{d-1}\times(0,\varepsilon)$  and define
$$h(x):= N[(R^{(d-1)}_d g) d\sigma](x)-N g(x).$$
Then $h$ has weak derivatives in $\R^d_+$ and for every $\phi\in C^\infty_c(\overline{\R^d_+})$,
\begin{equation}\label{eqneumann}
\int_{\R^d_+}\nabla \phi \cdot \nabla h \, dm=\int_{\R^d_+} \phi {g} \, dm.
\end{equation}

Furthermore, if $B_1$ has radius $r_1$ then for every $x\in\R^d_+\setminus B_1$ we have
\begin{equation}\label{eqboundhpointwise}
|h(x)|\lesssim \begin{cases}
\dfrac{1}{|x|^{d-2}}\norm{g}_1& \mbox{if }d>2,\\
\left(\left|\log|x|\right|+1 + r_1\dfrac{x_2 |\log x_2|}{|x|^2} \right) \norm{g}_1 & \mbox{if }d=2,\\
\end{cases} 
\end{equation}
and
\begin{equation}\label{eqboundpartialhpointwise}
|\nabla h(x)|\lesssim
\frac{1}{|x|^{d-1}}\left(1+\left|\log\frac{x_d}{|x|}\right|\right) \norm{g}_1.
\end{equation}
\end{lemma}
\begin{rem}
Note that $h$ can be understood as a weak solution to the Neumann problem
\begin{equation*}
\begin{cases}
-\Delta h (x)= g(x) & \mbox{if }x\in \R^d_+,\\
\partial_d h (y)= 0 & \mbox{if }y \in \partial\R^d_+.
\end{cases}
\end{equation*}
\end{rem}

\begin{proof}[Sketch of the proof of Lemma \ref{lemNeumann}]
Let us define $F$ as in \rf{eqF}. Then, 
\begin{equation*}
\nabla F = R^{(d-1)}[(R^{(d-1)}_d {g} )d\sigma]
\end{equation*}
and $h=F-N {g}$. It is an exercise to check that $F$ and $Ng$ are  $C^1$ up to the boundary, with $\partial_d F(y)=R^{(d-1)}_d{g}(y)$ for all $y\in \partial\R^d_+$. Consider $\phi\in C^\infty_c(\overline{\R^d})$.
Using the Green identities, since $F$ is harmonic in $\R^d_+$, we have
\begin{align*}
\int_{\R^d_+} \nabla \phi \cdot \nabla F \, dm - \int_{\R^d_+} \nabla \phi \cdot \nabla N {g} \, dm
	& = \int_{\partial\R^d_+} \phi \, \partial_d F \, d\sigma-  \int_{\partial\R^d_+} \phi\, R^{(d-1)}_d {g}  \, d\sigma + \int_{\R^d_+} \phi {g} \, dm =\int_{\R^d_+} \phi {g},
\end{align*}
proving \rf{eqneumann}.

To prove the pointwise bounds for $\nabla h$, recall that 
$$\nabla h (x)= R^{(d-1)} [(R^{(d-1)}_d g) d\sigma](x)-R^{(d-1)}g(x).$$
Given $x\in \R^d_+\setminus B_1$, since $\supp(g)\subset \frac14 B_1$,
\begin{equation}\label{eqacotaRieszmoltpetardeo}
|R^{(d-1)}g(x)|=c\left| \int_{B_1} \frac{g(z)(x-z)}{|x-z|^d} dz\right|\lesssim \frac{\norm{g}_1}{|x|^{d-1}}.
\end{equation}
On the other hand, consider $z\in \supp(g)\subset \frac14 B_1$ and $x\notin B_1$. Then, for $y\in\partial\R^d_+\cap B(0,|x|/2)$ one has $|x-y|\approx |x|$, for $y\in\partial\R^d_+\cap B(0,2|x|)\setminus B(0,|x|/2)$ one has $|y-z|\approx |x|$ and otherwise $|y-x|\approx|y-z|\approx|y|$. Thus, 
\begin{align}\label{eqacotaRieszpetardeo}
\nonumber\left|R^{(d-1)} [(R^{(d-1)}_d g) d\sigma](x)\right| 
	& = c\left|\int_{\partial\R^d_+}\left(\int_{B_1} \frac{g(z)z_d \,dz}{|y-z|^d} \right)\frac{(x-y) d\sigma(y)}{|x-y|^d} \right|\\
\nonumber	& \lesssim \int_{\partial\R^d_+\cap B(0, |x|/2)} \left(\int_{B_1} \frac{|g(z)| z_d \,dz}{|y-z|^d} \right)\frac{d\sigma(y)}{|x|^{d-1}} \\
\nonumber	& \quad + \int_{\partial\R^d_+\cap  B(0,2|x|)\setminus B(0,|x|/2)}\left(\int_{B_1} \frac{|g(z)|z_d \,dz}{|x|^d} \right) \frac{ d\sigma(y)}{|x-y|^{d-1}} \\
	& \quad+ \int_{\partial\R^d_+\setminus B(0, 2|x|)}\left(\int_{B_1} |g(z)|z_d \,dz \right)\frac{ d\sigma(y)}{|y|^{2d-1}}.
\end{align}

The first term can be bounded by $C\frac{\norm{g}_1}{|x|^{d-1}}$ because $\int_{\partial \R^d_+} \frac{d\sigma(y)}{|y-z|^d}= C \frac{1}{z_d}$. The second can be bounded by $C\frac{r_1\norm{g}_1}{|x|^{d}}\left|\log{\frac{|x_d|}{|x|}}\right|$ using polar coordinates and the last one can be bounded by $C \frac{r_1\norm{g}_1}{|x|^{d}}$ trivially. Thus, 
$$\left|R^{(d-1)} [(R^{(d-1)}_d g) d\sigma](x)\right|  \lesssim \frac{\norm{g}_1}{|x|^{d-1}}+ \frac{r_1\norm{g}_1}{|x|^{d}}\left|\log{\frac{x_d}{|x|}}\right|+ \frac{r_1\norm{g}_1}{|x|^{d}} ,$$
proving \rf{eqboundpartialhpointwise} since $r_1\leq |x|$. 

To prove the pointwise bounds for $h$, recall that 
$$h(x)= N [(R^{(d-1)}_d g) d\sigma](x)-Ng(x).$$
When $d>2$ we use the same method as in \rf{eqacotaRieszmoltpetardeo} and \rf{eqacotaRieszpetardeo} using Newton's potential instead of the vectorial $(d-1)$-dimensional Riesz transform to get 
$$|h^s(x)|\lesssim \frac{\norm{g}_1}{|x|^{d-2}}+\frac{r_1\norm{g}_1}{|x|^{d-1}}.$$ 

When $d=2$ the Newton potential is logarithmic, but the spirit is the same. In this case, arguing as before, 
$$|h^s(x)|\lesssim \log |x| \norm{g}_1+r_1\norm{g}_1 \frac{|x|+|x|\log |x|+x_2\log x_2}{|x|^2}.$$
\end{proof}

\begin{propo}\label{propomesuradecarleson}
Let $1<p<\infty$. Given a window $\mathcal{Q}$ of a special Lipschitz domain $\Omega$ with a Whitney covering $\mathcal{W}$ and given $f\in W^{1,p}(\Omega)$, define the \emph{Whitney averaging function}
\begin{equation}\label{eqdeftildeT}
{\mathcal{A}} f(x):=\sum_{Q\in\mathcal{W}} \chi_{Q}(x) \fint_{3Q} f(y) dy.
\end{equation}
If  $\mu$ is a finite positive Borel measure supported on $\delta_0 \mathcal{Q}$ with 
\begin{equation}\label{eqcotamesura}
\mu({\mathbf{Sh}}(Q))\leq C\ell(Q)^{d-p} \mbox{\,\,\,\, for every Whitney cube } Q\subset \mathcal{Q},
\end{equation}
and ${\mathcal{A}}: W^{1,p}(\Omega) \to L^p(\mu)$ is bounded, then $\mu$ is a $p$-Carleson measure.
\end{propo}

\begin{proof}
We will argue by duality. Let us assume that the window $\mathcal{Q}=Q(0,\frac{R}2)$ is of side-length $R$ and centered at the origin, which belongs to $\partial \Omega$.
Note that the boundedness of $\mathcal{A}$ is equivalent to the boundedness of its dual operator
$$\mathcal{A}^*: L^{p'}(\mu) \to (W^{1,p}(\Omega))^*.$$
We also assume that $\mu\equiv 0$ in a neighborhood of $\partial \Omega$. One can prove the general case by means of truncation and taking limits since the constants of the Carleson condition \rf{eqcarlesoncondition} and the the norm of the averaging operator will not get worse by this procedure.

Fix a cube $P$. Analogously to \cite[Theorem 3]{ars}, we apply the boundedness of $\mathcal{A}^*$ to the test function $g=\chi_{{\mathbf{Sh}}(P)}$ to get 
\begin{equation*}
\norm{\mathcal{A}^* g}_{(W^{1,p}(\Omega))^*}^{p'} \lesssim \norm{g}^{p'}_{L^{p'}(\mu)}= \mu({\mathbf{Sh}}(P)).
\end{equation*}
Thus, it is enough to prove that
\begin{equation}\label{eqcotaobjectiu}
\sum_{Q\leq P} \mu({\mathbf{Sh}}(Q))^{p'} \ell(Q)^{\frac{p-d}{p-1}} \lesssim \norm{\mathcal{A}^* g}_{(W^{1,p}(\Omega))^*}^{p'} + \mu({\mathbf{Sh}}(P)).
\end{equation}

Given any $f \in W^{1,p}(\Omega)$, using \rf{eqdeftildeT} and Fubini's Theorem,
$$\langle \mathcal{A}^*g, f \rangle =\int g \, {\mathcal{A}}f \, d\mu=\int_{\Omega} f \left(\sum_{Q\in\mathcal{W} }\frac{\chi_{3Q} }{m(3Q)}\int_Q g \, d\mu \right)\, dm,$$
where we wrote $\langle\cdot ,\cdot \rangle$ for the duality pairing. Consider 
\begin{equation}\label{eqdefgtilde}
{\widetilde{g}}(x):= \sum_{Q\in\mathcal{W}} \frac{\chi_{3Q}(x)}{m(3Q)}\int_Q g \, d\mu= \sum_{Q\leq P} \chi_{3Q}(x)\frac{\mu(Q)}{m(3Q)}.
\end{equation}
Then, 
\begin{equation*}
\langle \mathcal{A}^*g, f \rangle =\int_{\Omega} f \, \widetilde{g} \, dm.
\end{equation*}
Note that $\widetilde{g}$ is in $L^\infty$ with norm depending on the distance from the support of $\mu$ to $\partial\Omega$ by \rf{eqcotamesura}, but the norm of $\widetilde{g}$ in $L^1$ is
\begin{equation*}
\norm{\widetilde{g}}_{L^1}= \mu(\mathbf{Sh}(P)).
\end{equation*}

\begin{figure}[h]
\centering 
\begin{subfigure}[b]{0.45\textwidth}
\begin{tikzpicture}[line cap=round,line join=round,>=triangle 45,x=0.7cm,y=0.7cm]
\clip(-1.48,-2.5) rectangle (6.31,4.4);
\draw[help lines,fill=zzttqq,fill opacity=0.5] {[smooth,samples=50,domain=1.5:2.0] plot(\x,{2.5-(sin((\x)*180/pi)/2+abs((\x-3)^2-9)/8)})} -- (2,1.55) {[smooth,samples=50,domain=2.0:1.5] -- plot(\x,{3-(sin((\x)*180/pi)/2+abs((\x-3)^2-9)/8)})} -- (1.5,1.16) -- cycle;
\draw[help lines,fill=zzttqq,fill opacity=0.1] {[smooth,samples=50,domain=1.5:2.0] plot(\x,{0-(0/1)*\x-1.5/1})} -- (2,1.55) {[smooth,samples=50,domain=2.0:1.5] -- plot(\x,{3-(sin((\x)*180/pi)/2+abs((\x-3)^2-9)/8)})} -- (1.5,-1.5) -- cycle;
\draw[smooth,samples=100,domain=-1.4815679189171584:6.312208406939038] plot(\x,{6-(sin(((\x))*180/pi)/2+abs(((\x)-3)^2-9)/8)});
\draw[smooth,samples=100,domain=-1.4815679189171584:6.312208406939038] plot(\x,{1-(sin(((\x))*180/pi)/2+abs(((\x)-3)^2-9)/8)});
\draw[smooth,samples=100,domain=-1.4815679189171584:6.312208406939038] plot(\x,{0.5-(sin(((\x))*180/pi)/2+abs(((\x)-3)^2-9)/8)});
\draw[smooth,samples=100,domain=-1.4815679189171584:6.312208406939038] plot(\x,{3-(sin(((\x))*180/pi)/2+abs(((\x)-3)^2-9)/8)});
\draw[smooth,samples=100,domain=-1.4815679189171584:6.312208406939038] plot(\x,{2.5-(sin(((\x))*180/pi)/2+abs(((\x)-3)^2-9)/8)});
\draw[smooth,samples=100,domain=-1.4815679189171584:6.312208406939038] plot(\x,{2-(sin(((\x))*180/pi)/2+abs(((\x)-3)^2-9)/8)});
\draw[smooth,samples=100,domain=-1.4815679189171584:6.312208406939038] plot(\x,{1.5-(sin(((\x))*180/pi)/2+abs(((\x)-3)^2-9)/8)});
\draw[smooth,samples=100,domain=-1.4815679189171584:6.312208406939038] plot(\x,{1-(sin(((\x))*180/pi)/2+abs(((\x)-3)^2-9)/8)});
\draw[smooth,samples=100,domain=-1.4815679189171584:6.312208406939038] plot(\x,{4-(sin(((\x))*180/pi)/2+abs(((\x)-3)^2-9)/8)});
\draw[smooth,samples=100,domain=-1.4815679189171584:6.312208406939038] plot(\x,{5-(sin(((\x))*180/pi)/2+abs(((\x)-3)^2-9)/8)});
\draw (0,6)-- (0,-1.5);
\draw (-1,5.55)-- (-1,-1.5);
\draw[smooth,samples=100,domain=-1.4815679189171584:-1] plot(\x,{(3.5-(sin(((\x))*180/pi)/2+abs(((\x)-3)^2-9)/8))});
\draw[smooth,samples=100,domain=0:4] plot(\x,{(3.5-(sin(((\x))*180/pi)/2+abs(((\x)-3)^2-9)/8))});
\draw[smooth,samples=100,domain=6:6.312208406939038] plot(\x,{(3.5-(sin(((\x))*180/pi)/2+abs(((\x)-3)^2-9)/8))});
\draw (2.5,2.61)-- (2.5,-1.5);
\draw (3.5,3.08)-- (3.5,-1.5);
\draw (-1.5,3.09)-- (-1.5,-1.5);
\draw (6.5,3.49)-- (6.5,-1.5);
\draw (1,4.95)-- (1,-1.5);
\draw (-0.5,2.83)-- (-0.5,-1.5);
\draw (2,4.55)-- (2,-1.5);
\draw (4.5,2.65)-- (4.5,-1.5);
\draw (3,4.8)-- (3,-1.5);
\draw (5.5,3.01)-- (5.5,-1.5);
\draw (4,5.38)-- (4,-1.5);
\draw (5,5.85)-- (5,-1.5);
\draw (6,6.14)-- (6,-1.5);
\draw (0.5,3.42)-- (0.5,-1.5);
\draw (1.5,2.66)-- (1.5,-1.5);
\draw[smooth,samples=100,domain=0.5:3.5] plot(\x,{(1.25-(sin(((\x))*180/pi)/2+abs(((\x)-3)^2-9)/8))});
\draw[smooth,samples=100,domain=-1.4815679189171584:-1] plot(\x,{(0.75-(sin(((\x))*180/pi)/2+abs(((\x)-3)^2-9)/8))});
\draw[smooth,samples=100,domain=0:4.5] plot(\x,{(0.75-(sin(((\x))*180/pi)/2+abs(((\x)-3)^2-9)/8))});
\draw[smooth,samples=100,domain=6:6.312208406939038] plot(\x,{(0.75-(sin(((\x))*180/pi)/2+abs(((\x)-3)^2-9)/8))});
\draw (3.25,0.44)-- (3.25,-1.5);
\draw (3.75,0.23)-- (3.75,-1.5);
\draw (0.25,0.7)-- (0.25,-1.5);
\draw (-1.25,0.34)-- (-1.25,-1.5);
\draw (0.75,0.67)-- (0.75,-1.5);
\draw (1.75,0.08)-- (1.75,-1.5);
\draw (1.25,0.28)-- (1.25,-1.5);
\draw (2.25,0.06)-- (2.25,-1.5);
\draw (2.75,0.19)-- (2.75,-1.5);
\draw (4.25,0.52)-- (4.25,-1.5);
\draw (-0.25,0.43)-- (-0.25,-1.5);
\draw (-0.75,0.21)-- (-0.75,-1.5);
\draw (4.75,0.26)-- (4.75,-1.5);
\draw (5.25,0.44)-- (5.25,-1.5);
\draw[smooth,samples=100,domain=-1.4815679189171584:1.25] plot(\x,{(0.25-(sin(((\x))*180/pi)/2+abs(((\x)-3)^2-9)/8))});
\draw[smooth,samples=100,domain=3:5.5] plot(\x,{(0.25-(sin(((\x))*180/pi)/2+abs(((\x)-3)^2-9)/8))});
\draw[smooth,samples=100,domain=6:6.312208406939038] plot(\x,{(0.25-(sin(((\x))*180/pi)/2+abs(((\x)-3)^2-9)/8))});
\draw[smooth,samples=100,domain=-1.4815679189171584:0.75] plot(\x,{(0-(sin(((\x))*180/pi)/2+abs(((\x)-3)^2-9)/8))});
\draw[smooth,samples=100,domain=3.5:6.312208406939038] plot(\x,{(0-(sin(((\x))*180/pi)/2+abs(((\x)-3)^2-9)/8))});
\draw[smooth,samples=100,domain=-1.25:0.5] plot(\x,{(-0.25-(sin(((\x))*180/pi)/2+abs(((\x)-3)^2-9)/8))});
\draw[smooth,samples=100,domain=4:6.312208406939038] plot(\x,{(-0.25-(sin(((\x))*180/pi)/2+abs(((\x)-3)^2-9)/8))});
\draw[smooth,samples=100,domain=-1:0.25] plot(\x,{(-0.5-(sin(((\x))*180/pi)/2+abs(((\x)-3)^2-9)/8))});
\draw[smooth,samples=100,domain=4.25:6.312208406939038] plot(\x,{(-0.5-(sin(((\x))*180/pi)/2+abs(((\x)-3)^2-9)/8))});
\draw[smooth,samples=100,domain=-0.5:0] plot(\x,{(-0.75-(sin(((\x))*180/pi)/2+abs(((\x)-3)^2-9)/8))});
\draw[smooth,samples=100,domain=4.75:6.312208406939038] plot(\x,{(-0.75-(sin(((\x))*180/pi)/2+abs(((\x)-3)^2-9)/8))});
\draw[smooth,samples=100,domain=5.5:6] plot(\x,{(-1-(sin(((\x))*180/pi)/2+abs(((\x)-3)^2-9)/8))});
\draw (5.75,0.07)-- (5.75,-1.5);
\draw [domain=-1.48:6.31] plot(\x,{(-1.5-0*\x)/1});
\draw (1.46,1.68) node[anchor=north west] {$Q_\omega$};
\draw (1.41,-1.51) node[anchor=north west] {$\mathbf{Sh}_\omega(Q)$};
\end{tikzpicture}
\end{subfigure}
\begin{subfigure}[b]{0.08\textwidth}
\begin{tikzpicture}[line cap=round,line join=round,>=triangle 45,x=0.7cm,y=0.7cm]
\clip(0,-2.5) rectangle (1,4.4);
\draw (-0.1,1.3) node[anchor=north west] {$\longrightarrow$};
\draw (0.1,1.6) node[anchor=north west] {$\omega$};
\end{tikzpicture}
\end{subfigure}
\begin{subfigure}[b]{0.45\textwidth}
\begin{tikzpicture}[line cap=round,line join=round,>=triangle 45,x=0.7cm,y=0.7cm]
\clip(-1.47,-0.3) rectangle (6.66,6.83);
\fill[line width=0pt,color=zzttqq,fill=zzttqq,fill opacity=0.5] (1.5,4.5) -- (2,4.5) -- (2,5) -- (1.5,5) -- cycle;
\draw[help lines,fill=zzttqq,fill opacity=0.1] {[smooth,samples=50,domain=1.5:2.0] plot(\x,{sin((\x)*180/pi)/2+abs((\x-3)^2-9)/8+0.3})} -- (2,5) {[smooth,samples=50,domain=2.0:1.5] -- plot(\x,{5})} -- (1.5,{sin((1.5)*180/pi)/2+abs((1.5-3)^2-9)/8+0.3}) -- cycle;
\draw[smooth,samples=100,domain=-1.4664688919423372:6.663202115363525] plot(\x,{sin(((\x))*180/pi)/2+abs(((\x)-3)^2-9)/8+0.3});
\draw (1.42,5.17) node[anchor=north west] {$Q$};
\draw (-3,4.5)-- (-3,6);
\draw (-2.75,5.5)-- (-2.75,4);
\draw (-2.5,6)-- (-2.5,3.5);
\draw (-2.25,3)-- (-2.25,4.5);
\draw (-2,2.75)-- (-2,7);
\draw (-1.75,2.25)-- (-1.75,4);
\draw (-1.5,6)-- (-1.5,2);
\draw (-1,7)-- (-1,1.5);
\draw (-1.25,1.75)-- (-1.25,3);
\draw (-0.75,1.5)-- (-0.75,2.5);
\draw (-0.5,1.25)-- (-0.5,5);
\draw (-0.25,1.25)-- (-0.25,2.5);
\draw (0,1.25)-- (0,4);
\draw (0.25,1.5)-- (0.25,3);
\draw (0.5,1.75)-- (0.5,4);
\draw (0.75,2)-- (0.75,3.5);
\draw (1,2.25)-- (1,4);
\draw (1.25,2.25)-- (1.25,3.5);
\draw (1.5,2.5)-- (1.5,4);
\draw (1.75,2.5)-- (1.75,3.5);
\draw (2,2.5)-- (2,4.5);
\draw (2.25,2.5)-- (2.25,3.5);
\draw (2.5,2.5)-- (2.5,4.5);
\draw (2.75,2.5)-- (2.75,3.5);
\draw (3,2.25)-- (3,4.5);
\draw (3.25,2.25)-- (3.25,3.5);
\draw (3.5,2)-- (3.5,4.5);
\draw (4,1.75)-- (4,4.5);
\draw (4.25,1.5)-- (4.25,3);
\draw (4.5,1.5)-- (4.5,5);
\draw (4.75,1.25)-- (4.75,2.5);
\draw (5,1.25)-- (5,3.5);
\draw (5.25,1.25)-- (5.25,2.5);
\draw (5.5,1)-- (5.5,3.5);
\draw (5.75,1)-- (5.75,2);
\draw (6,1)-- (6,3.5);
\draw (6.25,1.25)-- (6.25,3);
\draw (6.5,1.75)-- (6.5,3.5);
\draw (6.75,2)-- (6.75,3.5);
\draw (7,2.25)-- (7,4.5);
\draw (7,4.5)-- (7,7);
\draw (6.5,6)-- (6.5,3.5);
\draw (6,3.5)-- (6,6);
\draw (5.5,5)-- (5.5,3.5);
\draw (5,6)-- (5,3.5);
\draw (5,6)-- (5,7);
\draw (6,6)-- (6,7);
\draw (5,6)-- (7,6);
\draw (5,7)-- (7,7);
\draw (7,5.5)-- (6,5.5);
\draw (7,5)-- (5,5);
\draw (7,4.5)-- (5,4.5);
\draw (5,4)-- (7,4);
\draw (7,3.5)-- (5,3.5);
\draw (7,3.25)-- (6.5,3.25);
\draw (5,3)-- (7,3);
\draw (7,2.75)-- (6,2.75);
\draw (7,2.5)-- (5,2.5);
\draw (7,2.25)-- (6,2.25);
\draw (6.75,2)-- (5,2);
\draw (5.5,2.25)-- (5,2.25);
\draw (6.5,1.75)-- (5,1.75);
\draw (6.25,1.5)-- (5,1.5);
\draw (6.25,1.25)-- (5,1.25);
\draw (6,1)-- (5.5,1);
\draw (5,1.25)-- (4.75,1.25);
\draw (5,1.5)-- (4.25,1.5);
\draw (5,1.75)-- (4,1.75);
\draw (5,2)-- (3.5,2);
\draw (5,2.25)-- (3,2.25);
\draw (5,2.5)-- (2.5,2.5);
\draw (4.5,2.75)-- (2.5,2.75);
\draw (4.5,3)-- (2.5,3);
\draw (3.5,3.25)-- (2.5,3.25);
\draw (4.5,3.5)-- (2.5,3.5);
\draw (5,3)-- (4.5,3);
\draw (5,3.5)-- (4.5,3.5);
\draw (5,4)-- (2.5,4);
\draw (5,4.5)-- (2.5,4.5);
\draw (5,5)-- (2.5,5);
\draw (4,5)-- (4,4.5);
\draw (4,5)-- (4,7);
\draw (5,7)-- (3,7);
\draw (2.5,6)-- (4,6);
\draw (4,5.5)-- (2.5,5.5);
\draw (3.5,4.5)-- (3.5,6);
\draw (3,4.5)-- (3,7);
\draw (2.5,4.5)-- (2.5,6);
\draw (3,7)-- (1,7);
\draw (2,7)-- (2,4.5);
\draw (2.5,6)-- (0.5,6);
\draw (2.5,5.5)-- (0.5,5.5);
\draw (2.5,5)-- (0.5,5);
\draw (2.5,4.5)-- (0.5,4.5);
\draw (2.5,4)-- (0.5,4);
\draw (2.5,3.5)-- (0.5,3.5);
\draw (2.5,3.25)-- (0.5,3.25);
\draw (1.5,4)-- (1.5,6);
\draw (1,7)-- (1,4);
\draw (3.75,3)-- (3.75,2);
\draw (2.5,2.5)-- (0,2.5);
\draw (2.5,2.75)-- (0,2.75);
\draw (2.5,3)-- (0,3);
\draw (0.5,3.5)-- (0,3.5);
\draw (0.5,4)-- (0,4);
\draw (1.25,2.25)-- (0,2.25);
\draw (0.75,2)-- (-1,2);
\draw (0.5,1.75)-- (-1,1.75);
\draw (0.25,1.5)-- (-1,1.5);
\draw (0,1.25)-- (-0.5,1.25);
\draw (0,2.25)-- (-1.25,2.25);
\draw (0,2.5)-- (-1.25,2.5);
\draw (0,3)-- (-1,3);
\draw (0,3.5)-- (-1,3.5);
\draw (0,4)-- (-1,4);
\draw (0.5,4)-- (0.5,6);
\draw (0,4)-- (0,7);
\draw (-1,6)-- (0.5,6);
\draw (0.5,5.5)-- (0,5.5);
\draw (0.5,5)-- (-1,5);
\draw (0.5,4.5)-- (-1,4.5);
\draw (-1,7)-- (1,7);
\draw (-1.25,1.75)-- (-1,1.75);
\draw (-1,2)-- (-1.5,2);
\draw (-1.25,2.25)-- (-1.75,2.25);
\draw (-1.25,2.5)-- (-1.75,2.5);
\draw (-1,2.75)-- (-2,2.75);
\draw (-1,3)-- (-2.25,3);
\draw (-1.5,3.25)-- (-2.25,3.25);
\draw (-1,3.5)-- (-2.5,3.5);
\draw (-1.5,3.75)-- (-2.5,3.75);
\draw (-1,4)-- (-2.75,4);
\draw (-2,4.25)-- (-2.75,4.25);
\draw (-1,4.5)-- (-3,4.5);
\draw (-3,4.75)-- (-2.5,4.75);
\draw (-1.5,5)-- (-3,5);
\draw (-2.5,5.25)-- (-3,5.25);
\draw (-1.5,5.5)-- (-3,5.5);
\draw (-1,6)-- (-3,6);
\draw (-1,5.5)-- (-1.5,5.5);
\draw (-1,5)-- (-1.5,5);
\draw (5,6)-- (4,6);
\draw (1.34,1.42) node[anchor=north west] {$\mathbf{Sh}(Q)$};
\begin{scriptsize}
\end{scriptsize}
\end{tikzpicture}
\end{subfigure}
\caption{We divide $\R^d_+$ in pre-images of Whitney cubes.}\label{figpseudocubes}
\end{figure}
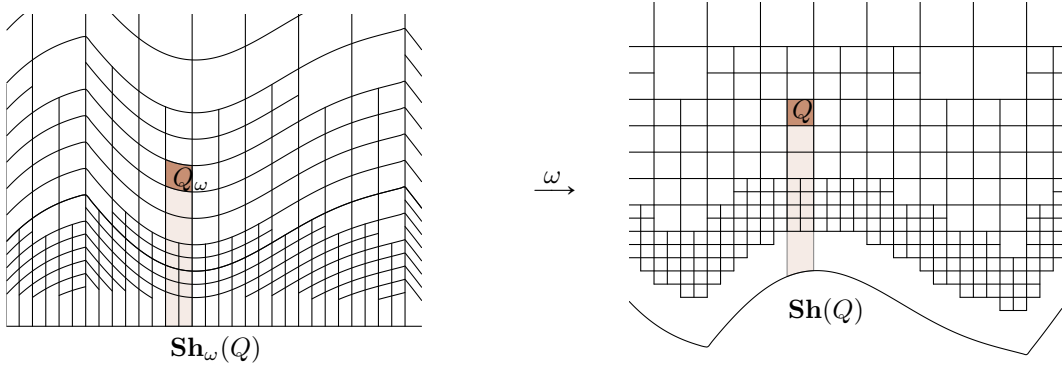

Consider also the change of variables $\omega: \R^d\to \R^d$, $\omega(x', x_d)= (x', x_d+A(x'))$ where $A$ is the Lipschitz function whose graph coincides with $\partial \Omega$, and to every Whitney cube $Q$ assign the set $Q_\omega=\omega^{-1}(Q)$ and its shadow $\mathbf{Sh}_\omega(Q)=\omega^{-1}({\mathbf{Sh}}(Q))$ (see Figure \ref{figpseudocubes}). Then, for every $x \in \R^d$ we define 
\begin{equation}\label{eqdefgtilde2}
{g}_0(x):={\widetilde{g}}(\omega(x))|{\rm det}(Dw(x))|,
\end{equation}
where ${\rm det}(Dw(\cdot))$ stands for the determinant of the jacobian matrix. Note that still $\norm{{g_0}}_{L^1}= \norm{\widetilde{g}}_{L^1}= \mu(\mathbf{Sh}(P))$, and
\begin{equation}\label{eqbrutalissim}
\langle \mathcal{A}^*g, f\rangle= \int_\Omega f\, \widetilde{g}\, dm  = \int_{\R^d_+} f\circ \omega\cdot {g}_0\, dm.
\end{equation}

The key of the proof is using 
\begin{equation}\label{eqdefh}
h(x):= N[(R^{(d-1)}_d g_0) d\sigma](x)-N g_0(x),
\end{equation}
which is the $L^1_{loc}(\R^d_+)$ solution of the Neumann problem 
\begin{equation}\label{eqneumann2}
\int_{\R^d_+}\nabla \phi \cdot \nabla h\, dm=\int_{\R^d_+} \phi \, {g}_0 \, dm \mbox{\quad\quad for every } \phi\in C^\infty_c(\overline{\R^d_+}),
\end{equation}
provided by Lemma \ref{lemNeumann}. 

We divide the proof in four claims.
\begin{claim}\label{clA*} If $\phi\in C^\infty_c (\overline{\R^d_+})$, then
\begin{equation*}
\langle \mathcal{A}^*g , \phi\circ \omega^{-1} \rangle=\int_{\R^d_+}  \nabla \phi \cdot \nabla h \, dm .
\end{equation*}
\end{claim}
\begin{proof}
Since $\omega$ is bilipschitz, the Sobolev $W^{1,p}$ norms before and after the change of variables $\omega$  are equivalent (see \cite[Theorem 2.2.2]{ziemer}). In particular, for $\phi\in C^\infty_c (\overline{\R^d_+})$, $\phi\circ \omega^{-1}\in W^{1,p}(\Omega)$ and we can use \rf{eqbrutalissim} and \rf{eqneumann2}. 
\end{proof}

Now we look for bounds for $\norm{\partial_d h}_{L^{p'}(\mathbf{Sh}_\omega(P))}$. The H\"older inequality together with a density argument would give us the bound
$$\norm{\mathcal{A}^*g}_{(W^{1,p}(\Omega))^*}\lesssim \norm{\nabla h}_{L^{p'}}+\mu(\mathbf{Sh}(P)),$$
with constants depending on the window size $R$, but we shall need a kind of converse. 

\begin{claim}\label{clacotagradienthk}
One has
\begin{align*}
\norm{\partial_d h}_{L^{p'}(\mathbf{Sh}_\omega(P))} \lesssim \norm{\mathcal{A}^*g}_{(W^{1,p}(\Omega))^*}+ \mu(\mathbf{Sh}(P)).
\end{align*}
\end{claim}
\begin{proof}
Take a ball $B_1$ containing $\omega^{-1}(4\mathcal{Q})$. The duality between $L^p$ and $L^{p'}$ gives us the bound
\begin{align*}
\norm{\partial_d h}_{L^{p'}(\mathbf{Sh}_\omega(P))}\lesssim \sup_{\substack{\phi\in C^\infty_c(B_1\cap\R^d_+)\\ \norm{\phi}_p\leq 1}} \left|\int \phi\, \partial_d h\,  dm \right|.
\end{align*}
To use the full potential of the Fourier transform, consider $h^s$ to be the symmetric extension of $h$ with respect to the hyperplane $x_d=0$, $h^s(x',x_d)=h(x', |x_d|)$. One can see that $h^s$ has global weak derivatives $\partial_j h^s=(\partial_j h)^s$ for $1\leq j\leq d-1$ and $\partial_d h^s(x', x_d)=-\partial_d h(x',-x_d)$ for every $x_d<0$.  Thus,
\begin{align}\label{eqdualitat}
\norm{\partial_d h}_{L^{p'}(\mathbf{Sh}_\omega(P))}\lesssim \sup_{\substack{\phi \in C^\infty_c(B_1)\\ \norm{\phi}_p\leq 1}}\left| \int \phi\, \partial_d h^s \, dm \right|.
\end{align}

Given $\phi \in C^\infty_c(B_1)$, consider the function $\widetilde{\phi} (x)= \phi(x)-\phi(x- 2\,r_1\, e_d)$, where $e_d$ denotes the unit vector in the $d$-th direction and $r_1=\frac12\diam(B_1)$, and take 
\begin{equation}\label{eqdefIphi}
I_\phi(x)=\int_{-\infty}^{x_d} \widetilde{\phi}(x',t) dt.
\end{equation}
Then, we have $I_\phi \in C^\infty_c(3B_1)$ with $\partial_d I_\phi \equiv {\phi}$ in the support of $\phi$ and $\norm{\partial_d I_\phi}_p^p= 2 \norm{\phi}_p^p$. Thus,
\begin{equation}\label{eqdoblasuport}
\int \phi\, \partial_d h^s dm = \langle \partial_d I_\phi, \partial_d h^s\rangle - \int_{3B_1\setminus B_1} \partial_d I_\phi \, \partial_d h^s dm 
\end{equation}
where we use the brackets for the dual pairing of test functions and distributions. Using H\"older's inequality and the estimate \rf{eqboundpartialhpointwise} one can see that the error term in \rf{eqdoblasuport} is bounded by
\begin{equation}\label{eqerrorxic}
 \int_{3B_1\setminus B_1} |\partial_d I_\phi  \, \partial_d h^s| dm \leq \norm{\partial_d I_\phi}_p \norm{\partial_d h^s}_{L^{p'}(3B_1\setminus B_1)}\leq C \norm{\phi}_{L^p} \mu(\mathbf{Sh}(P)) .
 \end{equation}
Note that $C$ only depends on $r_1$, which can be expressed as a function of the Lipschitz constant $\delta_0$ and the window side-length $R$. 

It is well known that the vectorial $d$-dimensional Riesz transform, 
$$R^{(d)}f(x)=\frac{1}{2w_{d+1}} {\rm p.v.} \int_{\R^d} \frac{x-y}{|x-y|^{d+1}} f(y)dy \mbox{ for every } f\in \mathcal{S}$$
is, in fact, a Calder\'on-Zygmund operator and, thus, it can be extended to a bounded operator in $L^p$.
Writing $R^{(d)}_{i}$ for the $i$-th component of the transform and $R^{(d)}_{ij}:=R^{(d)}_i\circ R^{(d)}_j$ for the double Riesz transform in the $i$-th and $j$-th directions, one has $\partial_{ii} I_\phi=R^{(d)}_{ii}\Delta I_\phi=\Delta R^{(d)}_{ii} I_\phi$ by a simple Fourier argument (see \cite[Section 4.1.4]{grafakos}). Thus, writing $f_\phi=R^{(d)}_{dd} I_\phi$, we have $\Delta f_\phi=\partial_{dd} I_\phi$, so
\begin{align}\label{eqamblesprimitives}
\langle \partial_d I_\phi, \partial_d h^s\rangle
	& = -\langle \partial_{dd} I_\phi, h^s\rangle 
	=   -\langle \Delta f_\phi, h^s\rangle.
\end{align}

Let $f_r= \varphi_r f_\phi$ with  $\varphi_r$ a bump function in $C^\infty_c(B_{2r}(0))$ such that $\chi_{B_r(0)}\leq \varphi_r\leq \chi_{B_{2r}(0)}$, $|\nabla \varphi_r|\lesssim 1/r$ and $|\Delta \varphi_r|\lesssim 1/r^2$. We claim that
\begin{align}\label{eqapareixriesz}
-\langle \Delta f_\phi, h^s\rangle = -\lim_{r\to \infty} \langle \Delta f_r, h^s\rangle=\lim_{r\to \infty} \langle \nabla f_r, \nabla h^s\rangle,
\end{align}
The advantage of $f_r$ is that it is compactly supported, while only the laplacian of $f_\phi$ is compactly supported.
Recall that $\Delta f_\phi=\partial_{dd} I_\phi \in C^\infty_c(\R^d)$ so, by the hypoellipticity of the Laplacian operator, $f_\phi \in C^\infty(\R^d)$ itself (see \cite[Corollary (2.20)]{folland}). Thus, the second equality in \rf{eqapareixriesz} comes from the definition of distributional derivative. It remains to prove
\begin{equation}\label{eqlaplaciares}
\langle \Delta f_r-\Delta f_\phi, h^s \rangle \xrightarrow{r\to\infty} 0.
\end{equation}
 
Since $\Delta f_\phi$ is compactly supported, taking $r$ big enough we can assume that
$$\Delta [(\varphi_r-1)f_\phi]=(\Delta \varphi_r) f_\phi + 2 \nabla \varphi_r \cdot \nabla f_\phi,$$
so
$$\left|\langle \Delta f_r-\Delta f_\phi, h^s \rangle\right|\lesssim \int_{B_{2r}(0)\setminus B_r(0)}\left( \frac{|f_\phi| |h^s|}{r^2}+\frac{|\nabla f_\phi| |h^s| }{r}\right) dm.$$
It is left for the reader to prove \rf{eqlaplaciares} plugging \rf{eqboundhpointwise} in this expression. One only needs to use that $f_\phi$ and $\nabla f_\phi$ are in every $L^q$ space for $1<q<\infty$.

Back to \rf{eqapareixriesz}, we can use $f_r^s(x',x_d):=f_r(x', -x_d)$ by a change of variables to obtain
\begin{align}\label{eqpassaaldual}
\int \nabla f_r \cdot \nabla h^s\,dm
	& = \int_{\R^d_+} \nabla f_r \cdot \nabla h\,dm+\int_{\R^d_+} \nabla f_r^s \cdot \nabla h \,dm =\langle \mathcal{A}^*g, (f_r+f_r^s)\circ \omega^{-1}\rangle
\end{align}
by means of Claim \ref{clA*}. Summing up, by \rf{eqdoblasuport}, \rf{eqerrorxic}, \rf{eqamblesprimitives}, \rf{eqapareixriesz} and \rf{eqpassaaldual} and letting $r$ tend to infinity, we get
\begin{align}\label{eqresum}
\left|\int \phi \, \partial_d h^s \,  dm\right| 
	& \lesssim \left|\langle \mathcal{A}^*g, (f_\phi+ f_\phi^s)\circ \omega^{-1}\rangle\right| + \norm{\phi}_{L^p}\mu(\mathbf{Sh}(P)) .
\end{align}

Using H\"older inequality in \rf{eqdefIphi} we have that $\norm{I_\phi}_p\leq C \norm{\phi}_p$.
Now, $\partial_j f_\phi= \partial_j R^{(d)}_{dd} I_\phi= R^{(d)}_{dj}\partial_d I_\phi$, so using the boundedness of the $d$-dimensional Riesz transform in $L^p$ we get
\begin{equation}\label{acotanormasobolev}
\norm{f_\phi}_{W^{1,p}}=\norm{f_\phi}_{L^p}+\norm{\nabla f_\phi}_{L^p}\leq C(\norm{I_\phi}_{p}+\norm{\partial_d I_\phi}_{p})\leq C \norm{\phi}_p .
\end{equation}

Summing up, by \rf{eqdualitat}, \rf{eqresum} and \rf{acotanormasobolev} we have got that
\begin{align*}
\norm{\partial_d h}_{L^{p'}(\mathbf{Sh}_\omega(P))}
	& \lesssim \sup_{\norm{f}_{W^{1,p}(\R^d)}\leq 1} \left|\langle \mathcal{A}^*g, f\circ \omega^{-1} \rangle\right|+ \mu(\mathbf{Sh}(P)).
\end{align*}

On the other hand, by \cite[Theorem 2.2.2]{ziemer} $\norm{f\circ \omega^{-1}}_{W^{1,p}(\Omega)}\approx \norm{f}_{W^{1,p}(\R^d_+)}$ for every $f$, so we have
\begin{align*}
\norm{\partial_d h}_{L^{p'}(\mathbf{Sh}_\omega(P))} \lesssim \sup_{\norm{f}_{W^{1,p}(\Omega)}\leq 1} \left|\langle \mathcal{A}^*g,f\rangle \right|+ \mu(\mathbf{Sh}(P)) = \norm{\mathcal{A}^*g}_{(W^{1,p}(\Omega))^*}+ \mu(\mathbf{Sh}(P)),
\end{align*}
that is Claim \ref{clacotagradienthk}.
\end{proof}

Next we stablish the relation between \rf{eqcotaobjectiu} and Claim \ref{clacotagradienthk}.
\begin{claim}\label{cl12} One has 
\begin{align} \label{eq12}
\nonumber\sum_{{Q\leq P}}  \mu(\mathbf{Sh}(Q))^{p'}\ell(Q)^{\frac{p-d}{p-1}}
	& \lesssim \norm{\partial_d h}_{L^{p'}(\mathbf{Sh}_\omega(P))}^{p'}  +\sum_{Q\leq P} \int_{Q_\omega} \left(\int_{\{z: z_d> x_d\}} \frac{z_d-x_d}{|x-z|^{d}} \widetilde{g}(\omega(z)) dz\right)^{p'}  dx\\
	& = \circled{\rm 1}+\circled{\rm 2}.
\end{align}
\end{claim}

\begin{proof}
Note that in \rf{eqdefh} we have defined $h$ in such a way that
 \begin{align*}
 \partial_d h(x)
	& = R^{(d-1)}_d[(R^{(d-1)}_d {g}_0) d\sigma](x)-R^{(d-1)}_d {g}_0(x)\\
	& = \frac{-1}{w_d}\int_{\R^d_+} \left(\frac{ 2 x_d z_d}{w_d } \int_{\partial \R^d_+}  \frac{d\sigma(y)}{|y-z|^{d} |x-y|^d} + \frac{x_d-z_d}{|x-z|^{d}}\right) {g}_0(z) dz.
\end{align*}
Given $x , z \in \R^d_+$, consider the kernel of $R^{(d-1)}_d[(R^{(d-1)}_d (\cdot)) d\sigma]-R^{(d-1)}_d (\cdot)$, 
\begin{align*}
G(x,z)
	& = \frac{ 2 x_d z_d}{w_d } \int_{\partial\R^d_+}  \frac{d\sigma(y)}{|y-z|^{d} |x-y|^d} + \frac{x_d-z_d}{|x-z|^{d}},
\end{align*} 
so that
\begin{align}\label{equsdelnucli}
 \partial_d h(x)
	& = \frac{-1}{w_d}\int_{\R^d_+} G(x,z){g}_0(z) \, dz.
\end{align}
We have the trivial bound
\begin{equation}\label{eqCotaDeG}
G(x,z)+\frac{z_d-x_d}{|x-z|^{d}}\, \chi_{\{z_d>x_d\}}(z) \geq 0,
\end{equation}
but given any Whitney cube $Q\leq P$, if $x \in Q_\omega$ and $z\in  {\mathbf{Sh}}_\omega(Q)$ we can improve the estimate. In this case,
$$\int_{\partial\R^d_+\cap \overline{{\mathbf{Sh}}_\omega(Q)}}  \frac{d\sigma(y)}{|y-z|^{d}}\gtrsim \int_{\partial\R^d_+\cap \overline{\omega^{-1}({\mathbf{Sh}}(\omega(z)))}}  \frac{d\sigma(y)}{|y-z|^{d}} \approx \frac{1}{z_d}$$
and, thus,
\begin{align}\label{eqCotaDeGMillorada}
\nonumber G(x,z)+\frac{z_d-x_d}{|x-z|^{d}}\, \chi_{\{z_d>x_d\}}(z)
	& \geq \frac{ 2 x_d z_d}{w_d } \int_{\partial\R^d_+}  \frac{d\sigma(y)}{|y-z|^{d} |x-y|^d} \\
	& \gtrsim \frac{ \ell(Q) z_d}{\ell(Q)^d } \int_{\partial\R^d_+\cap \overline{{\mathbf{Sh}}_\omega(Q)}}  \frac{d\sigma(y)}{|y-z|^{d}} \gtrsim  \frac{ \ell(Q)}{\ell(Q)^d }.
\end{align}

By the Lipschitz character of $\Omega$ we know that $|\det{D\omega}(z)|\approx 1$ for every $z\in\R^d_+$. Thus, by \rf{eqdefgtilde} and \rf{eqdefgtilde2}, given $Q\leq P$ we have
\begin{equation*}
\mu(\mathbf{Sh}(Q))=\sum_{{S\leq Q}}\mu(S) \lesssim \int_{{\mathbf{Sh}}(Q)} {\widetilde{g}}(w)dw \approx  \int_{{\mathbf{Sh}}_\omega(Q)} g_0(z)dz  .
\end{equation*}
For every $x\in Q_\omega$, using \rf{eqCotaDeG} and \rf{eqCotaDeGMillorada} first and then \rf{equsdelnucli} we get 
\begin{align*}
\mu(\mathbf{Sh}(Q))
	& \lesssim \int_{\R^d_+} G(x,z) g_0(z) \, dz\,  \ell(Q)^{d-1} + \int_{\{z : z_d> x_d\}} \frac{z_d-x_d}{|x-z|^{d}} \widetilde{g}(\omega(z)) \, dz\,  \ell(Q)^{d-1}\\
	& \lesssim |\partial_d h (x)| \ell(Q)^{d-1} + \int_{\{z : z_d> x_d\}} \frac{z_d-x_d}{|x-z|^{d}} \widetilde{g}(\omega(z)) \,dz\,  \ell(Q)^{d-1}.
\end{align*}
Then, raising to the power $p'$, averaging with respect to $x\in Q_\omega$ and summing with respect to $Q\leq P$ with weight $\rho_\mathcal{W}(Q)=\ell(Q)^{\frac{p-d}{p-1}}$, since $(d-1)p' +\frac{p-d}{p-1}-d=0$, we get Claim \ref{cl12}.

\end{proof}

Finally, we bound the negative contribution of the $(d-1)$-dimensional Riesz transform in \rf{eq12}, that is we bound \circled{2}. 
\begin{claim}\label{clacotariesz}
One has
\begin{align}\label{eqbound1}
\circled{{\rm 2}}= \sum_{Q\leq P} \int_{Q_\omega} \left(\int_{\{z: z_d> x_d\}} \frac{z_d-x_d}{|x-z|^{d}} \widetilde{g}(\omega(z)) dz\right)^{p'}  dx
	&  \lesssim \mu({\mathbf{Sh}}(P)).
\end{align}
\end{claim}

\begin{proof}
Consider $x, z \in \R^d_+$ with $x_d<z_d$ and two Whitney cubes $Q$ and $S$ such that $x\in Q_\omega$ and $z \in \omega^{-1}(3S)\setminus \omega^{-1}(3Q)$, then 
$$\frac{z_d-x_d}{|x-z|^{d}} \lesssim \frac{\dist(\omega(z), \partial\Omega)}{\Dist(S,Q)^d}\approx \frac{\ell(S)}{\Dist(S,Q)^d}.$$ 
On the other hand, when $3S \cap 3Q \neq \emptyset$, 
$$\int_{\omega^{-1}(3Q)}\frac{|z_d-x_d|}{|x-z|^{d}}\, dz \lesssim \ell(Q)\approx \ell(S).$$ 
From the definition of $\widetilde{g}$ in \rf{eqdefgtilde} it follows that ${\widetilde{g}}(\omega(z))\lesssim \sum_{L\in \mathcal{W}}\chi_{3L}(\omega(z)) \frac{\mu(L)}{m(3L)}$. Bearing all these considerations in mind, one gets
\begin{align*}
\circled{2}
	  \lesssim \sum_{Q\leq P} \ell(Q)^d \left(\sum_{S\leq P} \frac{\mu(S) \ell(S)}{\Dist(S,Q)^{d}}\right)^{p'} .
\end{align*}

Consider a fixed $\epsilon>0$.  One can apply first the H\"older inequality and then  \rf{eqcotamesura} to get 
\begin{align*}
\circled{2}
	&  \lesssim \sum_{Q\leq P} \ell(Q)^d \left(\sum_{S\leq P} \frac{\mu(S) \ell(S)^{1-\epsilon p'}}{\Dist(S,Q)^{d}}\right)\left(\sum_{S\leq P} \frac{\mu(S) \ell(S)^{1+\epsilon p}}{\Dist(S,Q)^{d}}\right)^{\frac{p'}{p}} \\
	&  \lesssim \sum_{Q\leq P} \ell(Q)^d \left(\sum_{S\leq P} \frac{\mu(S) \ell(S)^{1-\epsilon p'}}{\Dist(S,Q)^{d}}\right)\left(\sum_{S\leq P} \frac{\ell(S)^{d-p+1+\epsilon p}}{\Dist(S,Q)^{d}}\right)^{\frac{p'}{p}} .
\end{align*}
By Lemma \ref{lembad-1}, the last sum is bounded by $C\ell(Q)^{-p+1+\epsilon p} $ with $C$ depending on $\epsilon$ as long as  $d > d-p+1+\epsilon p>d-1$, that is, when $\frac{p-2}{p}<\epsilon<\frac{p-1}p$.
Thus, 
\begin{align*}
\circled{2}
	&  \lesssim \sum_{Q\leq P} \sum_{S\leq P} \frac{\mu(S) \ell(S)^{1-\epsilon p'}\ell(Q)^{d+(\epsilon -1)p'+p'/p}}{\Dist(S,Q)^{d}} \\
	&  = \sum_{S\leq P} \mu(S) \ell(S)^{1-\epsilon p'} \sum_{Q\leq P} \frac{\ell(Q)^{d-1+\epsilon p' }}{\Dist(S,Q)^{d}} .
\end{align*}
Again by Lemma \ref{lembad-1}, the last sum does not exceed $C \ell(S)^{-1+\epsilon p'}$ with $C$ depending on $\epsilon$ as long as $d> d-1+\epsilon p' >d-1$, that is when $0<\epsilon<\frac{1}{p'}=\frac{p-1}p$. Summing up, we need
$$\max\left\{\frac{p-2}{p},0\right\}<\epsilon<\frac{p-1}p.$$
 Such a choice of $\epsilon$ is possible for every $p>1$. Thus, 
\begin{align*}
\circled{2}
	&  \lesssim \sum_{S\leq P} \mu(S) =\mu({\mathbf{Sh}}(P)).
\end{align*}
\end{proof}

Now we can finish the proof of Proposition \ref{propomesuradecarleson}. The first term in \rf{eq12} is bounded due to Claim \ref{clacotagradienthk} by
\begin{equation}\label{eqacotagradientdeltot}
\circled{1}=\norm{\partial_d h}_{L^{p'}(\mathbf{Sh}_\omega(P))}^{p'}\lesssim \norm{\mathcal{A}^*g}_{(W^{1,p}(\Omega))^*}^{p'}+ \mu(\mathbf{Sh}(P))^{p'}.
\end{equation}
Being $\mu$ a finite measure, $\mu(\mathbf{Sh}(P))^{p'}\leq \mu(\mathbf{Sh}(P))\mu(\delta_0\mathcal{Q})^{p'-1}$ and, thus, the bounds \rf{eqbound1} and \rf{eqacotagradientdeltot} combined with \rf{eq12} prove \rf{eqcotaobjectiu}, leading to
$$\sum_{Q\leq P} \mu(\mathbf{Sh}(Q))^{p'}\ell(Q)^{\frac{p-d}{p-1}}
	 \lesssim \mu({\mathbf{Sh}}(P)).$$
\end{proof}

For the sake of clarity, we restate Theorem \ref{theoTb} in terms of Carleson measures.
\begin{theorem}\label{theomain2}
Given a Calder\'on-Zygmund smooth operator of order 1, a Lipschitz domain $\Omega$ and $1<p<\infty$, the following statements are equivalent:
\begin{enumerate}
\item Given any window $\mathcal{Q}$ with a properly oriented Whitney covering, and given any Whitney cube $P\subset \delta_0 \mathcal{Q}$, one has
$$\sum_{Q\leq P} \left(\int_{{\mathbf{Sh}}(Q)} |\nabla T_\Omega  (\chi_\Omega) |^p \, dm \right)^{p'} \ell(Q)^{\frac{p-d}{p-1}}\leq C \int_{{\mathbf{Sh}}(P)} |\nabla T_\Omega  (\chi_\Omega) |^p \, dm.$$
\item $T_\Omega $ is a bounded operator on $W^{1,p}(\Omega)$.
\end{enumerate}
\end{theorem}

\begin{proof}
The implication $1 \implies 2$ is Theorem \ref{theocarleson}. 

To prove that $2 \implies 1$ we will use the previous proposition. Let us assume that we have a properly oriented Whitney covering $\mathcal{W}$ associated to an $R$-window $\mathcal{Q}$ of a Lipschitz domain $\Omega$, where we assume that the window $\mathcal{Q}=Q(0,\frac{R}2)$ is of side-length $R$ and centered at the origin.
Note that since $T_\Omega$ is bounded in $W^{1,p}(\Omega)$ then, by the Key Lemma, 
\begin{equation}\label{eqBoundSumaMu}
\sum_{Q\in\mathcal{W}} \left|\fint_{3Q} f \, dm\right|^p \int_Q|\nabla T_\Omega (\chi_\Omega)(x)|^p \, dx \lesssim \norm{f}_{W^{1,p}(\Omega)}.
\end{equation}

Consider the Lipschitz function $A:\R^{d-1}\to \R$ whose graph coincides with the boundary of $\Omega$ in $\mathcal{Q}$. We say that ${\widetilde{\Omega}}$ is the special Lipschitz domain defined by the graph of $A$ that coincides with $\Omega$ in the window $\mathcal{Q}$. One can consider a Whitney covering $\tilde{\mathcal{W}}$ associated to $\widetilde{\Omega}$ such that it coincides with $\mathcal{W}$ in $\delta_0 \mathcal{Q}$. Consider the averaging operator
$${\mathcal{A}}f(x):=\sum_{Q\in\tilde{\mathcal{W}}} \chi_{Q}(x) \fint_{3Q} f(y) \, dy \mbox{\quad for } f\in W^{1,p}(\widetilde{\Omega}). $$
Writing $d\mu(x):=|\nabla T(\chi_\Omega)(x)|^p\, \chi_{\delta_0 \mathcal{Q}}(x) \,dx$, it is easy to see that \rf{eqBoundSumaMu} implies the boundedness of 
 $${\mathcal{A}}: W^{1,p}({\widetilde{\Omega}}) \to L^p(\mu)$$
(consider an appropriate bump function and use the Leibnitz formula). 

In order to apply Proposition \ref{propomesuradecarleson}, we only need to show that $\mu({\mathbf{Sh}}(Q))\leq C\ell(Q)^{d-p}$ for every Whitney cube $Q\subset \mathcal{Q}$, which in particular implies that $\mu$ is finite.
Consider a bump function $\varphi_Q$ such that $\chi_{{\mathbf{Sh}}(2Q)}\leq \varphi_Q \leq \chi_{{\mathbf{Sh}}(4Q)}$ with $|\nabla \varphi_Q| \lesssim \frac1{\ell(Q)}$.

Then, 
\begin{align*}
\mu({\mathbf{Sh}}(Q))
	& = \int_{\mathbf{Sh}(Q)\cap \delta_0\mathcal{Q}}|\nabla T \chi_{\Omega} (x)|^p dx \leq \int_{\mathbf{Sh}(Q)}|\nabla T (\chi_{\Omega} - \varphi_Q) (x)|^p dx + \int_{\Omega} |\nabla T \varphi_Q (x)|^p dx .
\end{align*}
With respect to the first term, notice that given $x\in\mathbf{Sh}(Q)$,   $\dist(x, \supp (\chi_{\Omega} - \varphi_Q)) > \frac12\ell(Q)$ so  Lemma \ref{lemderivaT} together with \rf{eqCZKderivades} allows us to write
$$|\nabla T  (\chi_{\Omega} - \varphi_Q) (x)|\lesssim \int_{\Omega\setminus \mathbf{Sh}(2Q)} \frac{1}{|y-x|^{d+1}}\, dy\lesssim \frac{1}{\ell(Q)}.$$
Being $\Omega$ a Lipschitz domain, $m(\mathbf{Sh}(Q))\approx \ell(Q)^d$, so 
$$ \int_{\mathbf{Sh}(Q)}|\nabla T (\chi_{\Omega} - \varphi_Q) (x)|^p dx\lesssim \ell(Q)^{d-p}.$$

The second term is bounded by hypothesis by a constant times $\norm{\varphi_Q}_{W^{1,p}(\Omega)}^p$, and 
$$\norm{\varphi_Q}_{W^{1,p}(\Omega)}^p\approx \norm{\varphi_Q}_{L^p(\Omega)}^p+\norm{\nabla\varphi_Q}_{L^p(\Omega)}^p\lesssim \ell(Q)^d + \ell(Q)^{d-p} \lesssim (R^p+1)\ell(Q)^{d-p}, $$
where $R$ is the side-length of the $R$-window $\mathcal{Q}$, proving that $\mu$ satisfies \rf{eqcotamesura}.
\end{proof}

\section{Final remarks}\label{secfinalremarks}
\begin{rem}
The article of Arcozzi, Rochberg and Sawyer \cite{ars} has been the cornerstone in our quest for necessary conditions related to Carleson measures. In fact their article provides a quick shortcut for the proof of Theorem \ref{theomain2} (avoiding Proposition \ref{propomesuradecarleson}) for simply connected domains of class $C^1$ in  the complex plane, and we believe it is worth to give a hint of the reasoning. 
\end{rem}

\begin{proof}[Sketch of the proof]
In the case of the unit disk, we found in the Key Lemma that if $T$ is a smooth convolution Calder\'on-Zygmund operator of order 1 bounded in $W^{1,p}(\DDD)$, then 
\begin{equation}\label{eqsumacubsdeguays}
\sum_{Q\in \mathcal{W}} \left|\fint_{3Q}f \, dm \right|^p \int_Q|\nabla T \chi_\DDD (z)|^p dm(z) \lesssim \norm{f}_{W^{1,p}(\DDD)}^p
\end{equation}
for all $f\in W^{1,p}(\DDD)$. 
If one considers $d\mu(z)=|\nabla T \chi_\DDD(z)|^p dm(z)$ and $\rho(z)=(1-|z|^2)^{2-p}$, then, when $f$ is in the Besov space of analytic functions on the unit disk $B_p(\rho)$,
$$\norm{f}_{B_p(\rho)}^p:=|f(0)|^p+\int_\DDD|f'(z)|^p(1-|z|^2)^p \rho(z) \frac{dm(z)}{(1-|z|^2)^2}\approx \norm{f}_{W^{1,p}(\DDD)}^p.$$
Using the mean value property (and \rf{eqcotamesura} for the error terms), one can see that if $T$ is bounded, then for every holomorphic function $f$ the bound in \rf{eqsumacubsdeguays} is equivalent to
\begin{equation*}
\int_\DDD |f(z)|^p |\nabla T \chi_\DDD(z)|^p dm(z) \lesssim \norm{f}_{B_p(\rho)}^p,
\end{equation*}
i.e., $\norm{f}_{L^p(\mu)}\lesssim \norm{f}_{B_p(\rho)}$. Following the notation in \cite{ars}, the measure $\mu$ is a Carleson measure for $(B_p(\rho), p)$, stablishing Theorem \ref{theomain2} for the unit disk by means of Theorem 1 in that article.

For $\Omega \subset \C$ Lipschitz and $f$ analytic in $\Omega$, we also have
\begin{align*}
\int_\Omega |f(z)|^p |\nabla T \chi_\Omega(z)|^p dm(z) \lesssim \norm{f}_{W^{1,p}(\DDD)}^p .
\end{align*}
If $\Omega$ is simply connected, considering a Riemann mapping $F:\DDD \to \Omega$, and using it as a change of variables, one can rewrite the previous inequality as
\begin{align*}
\int_\DDD  |f\circ F|^p \mu(F(\omega)) |F'(\omega)|^2 dm(\omega)
	& \lesssim |f(F(0))|^p  + \int_\DDD |(f\circ F)'(\omega)|^p|F'(\omega)|^{2-p} dm(\omega).
\end{align*}
Writing $d\tilde{\mu}(\omega)=\mu(F(\omega)) |F'(\omega)|^2 dm(\omega)$, and $\rho(\omega)=|F'(\omega)(1-|\omega|^2)|^{2-p}$, one has that given any $g$ analytic on $\DDD$, 
$$\norm{g}_{L^p(\tilde\mu)}\lesssim \norm{g}_{B_p(\rho)}.$$

So far so good, we have seen that $\tilde\mu$ is a Carleson measure for $(B_p(\rho),p)$, but we only can use \cite[Theorem 1]{ars} if
 two conditions on $\rho$ are satisfied. The first condition is that the weight $\rho$ is ``almost constant'' in Whitney squares, that is
 $$\mbox{for }x_1, x_2\in Q\in\mathcal{W} \implies \rho(x_1)\approx\rho(x_2),$$ and this is a consequence of Koebe distortion theorem, which asserts that for every $w\in\DDD$ we have
 $$|F'(\omega)| (1-|\omega|^2) \approx \dist(F(\omega), \partial \Omega)$$
 (see \cite[Theorem 2.10.6]{aim}, for instance). The second condition is the Bekoll\'e-Bonami condition, which is 
$$\int_Q (1-|z|^2)^{p-2}\rho(z) dm(z) \left( \int_Q \left( (1-|z|^2)^{p-2}\rho(z)\right)^{1-p'} dm(z)\right)^{p-1}\lesssim m(Q)^{p}.$$
If the domain $\Omega$ is Lipschitz with small constant depending on $p$ (in particular if it is $C^1$), then this condition is satisfied (see \cite[Theorem 2.1]{bekolle}).
\end{proof}

\begin{rem}
Quite likely, our arguments to prove the Key Lemma apply to more general domains, such as the so called uniform domains. However, for simplicity, we only deal with Lipschitz domains in this paper and  we do not pursue the objective of extending Theorem \ref{theoTP} to more general Sobolev extension domains.

We want to point out some open problems to conclude this exposition. First of all, when $n>1$ we have found a sufficient condition in terms of Carleson measures, but we do not know if this condition (or a similar one) is necessary.

Secondly, it would be interesting to study the fractional Sobolev spaces, $W^{s,p}(\Omega)$ for $s\notin \N$.

Finally we have obtained some results connecting the boundedness of the even smooth convolution Calder\'on-Zygmund operators to the geometry of the boundary of planar domains $\Omega$ which will be published in a forthcoming paper.
\end{rem}

\renewcommand{\abstractname}{Acknowledgements}
\begin{abstract}
The authors were funded by the European Research
Council under the European Union's Seventh Framework Programme (FP7/2007-2013) /
ERC Grant agreement 320501. Also, partially supported by grants 2014-SGR-75 (Generalitat de Catalunya), MTM2010-16232 and MTM2013-44304-P (Spanish government). The first author was also funded by a FI-DGR grant from the Generalitat de Catalunya, (2014FI-B2 00107).\end{abstract}

\end{document}